%% file: template1.tex
\documentclass{article}
\usepackage{mathtools}  
\usepackage{xfrac} 
\usepackage{amsthm}
\usepackage{arxiv}
\usepackage{amsmath}
\usepackage{natbib}
\usepackage{wasysym}
\usepackage[utf8]{inputenc} % allow utf-8 input
\usepackage[T1]{fontenc}    % use 8-bit T1 fonts
\usepackage{hyperref}       % hyperlinks
\usepackage{url}            % simple URL typesetting
\usepackage{booktabs}       % professional-quality tables
\usepackage{amsfonts}       % blackboard math symbols
\usepackage{nicefrac}       % compact symbols for 1/2, etc.
\usepackage{microtype}      % microtypography
\usepackage{lipsum}
\usepackage{graphicx}

\newtheorem{theorem}{Theorem}[section]

\newtheorem{lemma}[theorem]{Lemma}

\graphicspath{ {./images/} }
\theoremstyle{plain}
\newtheorem{asu}{Assumption}
\title{On Strong Consistency and Asymptotic Normality of Least Absolute Deviation Estimators for 2-dimensional Sinusoidal Model}

\author{
	$ \text{Saptarshi Roy}^{1}$, $\text{Amit Mitra}^{2*}$ and $\text{N K Archak}^{2}$\\
	$^{1}$ Department of Statistics, Texas A\&M University, College Station, Texas 77840, USA\\
	$^{2}$ Department of Mathematics and Statistics, Indian Institute of Technology Kanpur, UP, India 208016\\
	$^{*}$ Corresponding author; Email: amitra@iitk.ac.in \\
	}

\begin{document}
\maketitle

\input{abstract}
\input{Introduction}

\input{Assumptions}
\input{SC_AN}
\input{SC_ANmultiple}
\input{Simulation}

\input{DataAnalysis}
\input{Conclusions}

\bibliographystyle{apalike}
\bibliography{references}
\input{appendix}

\end{document}

%% file: abstract.tex
\begin{abstract}
 Estimation of the parameters of a 2-dimensional sinusoidal model is a fundamental problem in digital signal processing and time series analysis.  In this paper, we propose a robust least absolute deviation (LAD) estimators for parameter estimation. The proposed methodology provides a robust alternative to non-robust estimation techniques like the least squares estimators, in situations where outliers are present in the data or in the presence of heavy tailed noise. We study important asymptotic properties of the LAD estimators and establish the strong consistency and asymptotic normality of the LAD estimators of the signal parameters of a 2-dimensional sinusoidal model. We further illustrate the advantage of using LAD estimators over least squares estimators through extensive simulation studies.  Data analysis of a 2-dimensional texture data indicates practical applicability of the proposed LAD approach.   

\keywords{: 2-dimensional sinusoidal model \and asymptotic normality\and heavy tailed noise\and image processing\and least absolute deviation estimator\and least squares estimator\and strong consistency\and texture analysis}
\end{abstract}

%% file: Introduction.tex
\section{Introduction}
Let us consider a superimposed 2-dimensional sinusoidal signal model, \begin{equation}
    y(t,s) = \sum\limits_{k=1}^{p}\Big(A_{k}^0 \cos(\lambda_{k}^0t + \mu_{k}^0s) + B_{k}^0\sin(\lambda_{k}^0t + \mu_{k}^0s)\Big) + \epsilon(t,s). \label{eqn:1}
\end{equation}
where $t \in [T], s \in [S]$, with $[T] = \{1, \cdots, T\}$. Here the signal $y(t,s)$ is decomposed into two components, the first term of right-hand side of \eqref{eqn:1} is the deterministic component and the second term is the random noise component. In this paper, we assume that the order of the above model $p$ is known. $A_{k}^0$s and $B_{k}^0$s are the unknown amplitudes and $\lambda_k^0, \mu_k^0 \in [0,\pi]$ are the unknown frequencies. Given the $TS$ signal observations, $y(1,1),\ldots, y(T,S)$, the problem is to estimate of the unknown parameters. 

2-dimensional sinusoidal models have been used in varied real life applications in the fields of image analysis and restoration, medical imaging, wireless communications, array processing, synthetic aperture radar imaging, among others. See for example, \cite{andrews}, \cite{rossi}, \cite{dudgeon}, \cite{Odendaal}, \cite{Van}, \cite{grover} and the references cited therein.  In particular, the 2-dimensional sinusoidal model has wide a variety of applications in texture analysis. \cite{francos} has shown that \eqref{eqn:1} can be used effectively in modelling texture images and proposed estimation of unknown frequencies by selecting the sharpest peaks of the periodogram function $I(\lambda, \mu)$ of the observed signals $y(t,s)$. The 2-dimensional version of the periodogram function is given by, $$I(\lambda, \mu) = \frac{1}{TS}\sum\limits_{t=1}^{T}\sum\limits_{s=1}^{S}\left|y(t,s)e^{-i(\lambda t + \mu s)}\right|^2.$$ 

Under the assumption that the number of components, $p$, is known, extensive work on model \eqref{eqn:1} or on it's variations have been carried out by several authors. \cite{rao} studied the consistency and asymptotic normality of maximum likelihood estimates of the parameters of 2-dimensional  superimposed exponential signals under normal noise distribution, \cite{kundu_mitra} established the asymptotic properties of least square estimators (LSE) for the same 2-dimensional exponential signals (see also \cite{kundugupta}), \cite{cohen} established the asymptotic properties of the LSE for a colored noise setup. \cite{mitra_stoica} derived the asymptotic Cramer-Rao lower bound (CRLB) of 2-dimensional superimposed exponential models. \cite{prasad_kundu} proposed a sequential LSE estimation procedure for parameter estimation of 2-dimensional Sinusoidal model. Recently, \cite{grover} proposed an efficient algorithm for parameter estimation, which achieves the optimal convergence rate as that of LSE in finite number of steps. Other prominent algorithms for parameter estimation of 2-dimensional sinusoidal model can be found in \cite{hua}, \cite{francos}, \cite{clark}, \cite{bansal}, \cite{zhang}, \cite{kundunandi}.

The least squares estimators have the optimal rate of convergence and achieve the CRLB asymptotically (\cite{kundu_mitra}, \cite{mitra_stoica}).  It may however be noted that, although LSE have theoretical asymptotic optimal properties, LSE or sequential LSE are well known to be non-robust under the presence of outliers  or in the presence of heavy tailed noise component. In this paper, we propose a robust Least Absolute Deviation(LAD) method of estimation for the parameters of (\ref{eqn:1}). Unlike the LSE or the sequential LSE, LAD approach gives a robust method of estimation as LAD gives equal weight to all the residuals, in contrast to LSE, which, by squaring the residuals, gives more weight to large residuals. \cite{kim} proposed LAD based estimation for a 1-dimensional superimposed sinusoidal model and studied the asymptotic properties of the LAD.  Our present work extends the work of \cite{kim} to the case of 2-dimensional  sinusoidal model. Nowhere in the literature, at least not known to the authors, any attempt has been made to study the theoretical asymptotic properties of the LAD estimators for a 2-dimensional sinusoidal model.

The rest of the paper is arranged as follows. The model assumptions and methodology are given in Section \ref{sec:2}. We derive the strong consistency and asymptotic normality of LAD  estimators for one harmonic component case, i.e. $p = 1$, in Section \ref{sec:3}. Section \ref{sec:4} gives the results for  multiple harmonic components model. Results of the simulation study to validate the asymptotic results and also to ascertain the robust performance of LAD estimators are discussed in Section \ref{sec:5}. Section \ref{sec:6} presents texture analysis of a synthetic grayscale image using the LAD approach. Finally, conclusions and implications of the present work are discussed in Section \ref{sec:7}.

%% file: Assumptions.tex
\section{LAD methodology and model assumptions}
\label{sec:2}

Consider the  model given by (\ref{eqn:1}). The LAD estimators of the true parameter vector $\mathbf{\theta_0} = (A_{1}^0,B_{1}^0,\lambda_{1}^0,\mu_{1}^0,\dots,A_{p}^0,B_{p}^0,\lambda_{p}^0,\mu_{p}^0)$, denoted by $\hat{\theta}_{T,S}$, is obtained by minimizing the objective function 
\begin{equation*}
    Q_{T,S}(\mathbf{\theta}) = \frac{1}{TS}\sum\limits_{t=1}^{T}\sum\limits_{s=1}^{S}\Big|y(t,s) - \sum\limits_{k=1}^{p}(A_{k}\cos(\lambda_{k}t + \mu_{k}s) + B_{k}\sin(\lambda_{k}t + \mu_{k}s)) \Big|, \label{eq:2}
\end{equation*}

\begin{equation*}
    \text{i.e.}\:\: \hat{\theta}_{T,S}=\arg\min Q_{T,S} (\mathbf{\theta}).
\end{equation*}
Where, $\mathbf{\theta} = (A_{1},B_{1},\lambda_{1},\mu_{1},\dots,A_{p},B_{p},\lambda_{p},\mu_{p}) \in \Theta $; $\Theta \subseteq \mathbb{R}^{4p}$. 

In order to derive the asymptotic properties of LAD estimators, we make the following two assumptions for the model  (\ref{eqn:1}).
\begin{asu}\label{as1} 
  $\mathbf{\theta} = (A_{1},B_{1},\lambda_{1},\mu_{1},\dots,A_{p},B_{p},\lambda_{p},\mu_{p}) \in \Theta $, where $\Theta$ denotes the parameter space with $\Theta = (K \times K \times [0,\pi] \times [0,\pi])^{p} $, $K$ being the compact subset of \emph{$\mathbb{R}$}. The true parameter $\theta_0$ is assumed to be an interior point of the parametric space $\Theta$.
\end{asu}

\begin{asu}\label{as2}
  $\epsilon(t,s)$ are independent and identically (i.i.d) random variables with common distribution function $G(.)$ and continuous probability density function $g(.)$ such that $G(.)$ has unique median at 0 i.e. $G(0) = \frac{1}{2}$. Moreover, let $g'(.)$ exists and is bounded. Further, we assume that second order moment for $\epsilon(t,s)$ is finite i.e. $\mathbb{E}[\epsilon(t,s)^{2}] < \infty~, \forall~ t,s$.
\end{asu}

%% file: SC_AN.tex
\section{Asymptotic results for one component model}\label{sec:3}
For simplicity, let us consider a one component model (i.e. $p=1$ in (\ref{eqn:1})).  In this section, we derive the strong consistency of LAD estimators for a one component model. The results for a  $p>1$ harmonic components model will be discussed in Section \ref{sec:5}. For $p=1$, (\ref{eqn:1}) is 
\begin{equation}
     y(t,s) = A^0 \cos(\lambda^0 t + \mu^0 s) + B^0 \sin(\lambda^0 t + \mu^0 s) + \epsilon(t,s)~~ t\in [T], s\in[S]. \label{eq:3}
\end{equation}
To derive the strong consistency and asymptotic normality results of LAD estimators under Assumption \ref{as1} and Assumption \ref{as2}, we need the following lemma.
\begin{lemma} \label{lemma:1}
If $(\theta_{1},\theta_{2})~\in~(0,\pi)\times(0,\pi)~ \text{and}~ t,s\in\{0,1,2\}$, then except for countable number of points the following are true,

\item[1.] \label{item-one}
    $\lim\limits_{T,S \rightarrow \infty} \frac{1}{T^{k_1+1}S^{k_2+1}} \sum\limits_{t=1}^{T} \sum\limits_{s=1}^{S} t^{k_1} s^{k_2} \cos^2(\theta_1t + \theta_2s) = \frac{1}{2(k_1+1)(k_2+1)}$ 

\item[2.] \label{item-two} $\lim\limits_{T,S \rightarrow \infty} \frac{1}{T^{k_1+1}S^{k_2+1}} \sum\limits_{t=1}^{T} \sum\limits_{s=1}^{S} t^{k_1} s^{k_2} \sin^2(\theta_1t + \theta_2s) = \frac{1}{2(k_1+1)(k_2+1)}$ 

\item[3.] \label{item-three} $\lim\limits_{T,S \rightarrow \infty} \frac{1}{T^{k_1+1}S^{k_2+1}} \sum\limits_{t=1}^{T} \sum\limits_{s=1}^{S} t^{k_1} s^{k_2} \cos(\theta_1t + \theta_2s) = 0$ 

\item[4.] \label{item-four} $\lim\limits_{T,S \rightarrow \infty} \frac{1}{T^{k_1+1}S^{k_2+1}} \sum\limits_{t=1}^{T} \sum\limits_{s=1}^{S} t^{k_1} s^{k_2} \sin(\theta_1t + \theta_2s) = 0$ 

\item[5.] \label{item-five} $\lim\limits_{T,S \rightarrow \infty} \frac{1}{T^{k_1+1}S^{k_2+1}} \sum\limits_{t=1}^{T} \sum\limits_{s=1}^{S} t^{k_1} s^{k_2} \sin(\theta_1t + \theta_2s)\cos(\theta_1t + \theta_2s) = 0$

\end{lemma}

\begin{proof}
    Follows from the result of  \cite{vinogradov}.
\end{proof}

The following theorem presents the strong consistency result of the LAD estimators.
\begin{theorem}\label{theorem:1}
  Under Assumption \ref{as1} and Assumption \ref{as2}, LAD estimator $\hat{\theta}_{T,S}$ is strongly consistent for $\theta_0$,  
   \begin{align*}
		\text{i.e.} \:\: \hat{\theta}_{T,S} \xrightarrow{a.s.}  {\theta}_0, \:\: \text{as} \:\: min\{T,S\} \rightarrow \infty.
	\end{align*}  
\end{theorem}

\begin{proof}
Please refer to Appendix \ref{Appendix:1} for the detailed exposition of the proof. 
\end{proof}

%\section{Asymptotic normality for one component model}\label{sec:4}
The following theorem gives the asymptotic normality property of LAD estimators.
\begin{theorem}\label{theorem:2}
If Assumption \ref{as1} and Assumption \ref{as2} hold for (1), then \\ $\left(\sqrt{TS}(\hat{A}_{T,S}-A^0),\sqrt{TS}(\hat{B}_{T,S}-B^0),T^{\frac{3}{2}}S^{\frac{1}{2}}(\hat{\lambda}_{T,S}-\lambda^0),S^{\frac{3}{2}}T^{\frac{1}{2}}(\hat{\mu}_{T,S}-\mu^0)\right)$ converges in distribution to $N_4 \left(0,\frac{1}{4g^{2}(0)}\Sigma^{-1}\right)$ as $min\{T,S\} \rightarrow \infty,$ where $$\Sigma = \begin{pmatrix}
\frac{1}{2} & 0 & \frac{B^{0}}{4} & \frac{B^{0}}{4}\\
0 & \frac{1}{2} & -\frac{A^{0}}{4} & -\frac{A^{0}}{4}\\
\frac{B^{0}}{4} & -\frac{A^{0}}{4} & \frac{{A^{0}}^{2} + {B^{0}}^{2}}{6} & \frac{{A^{0}}^{2}+{B^{0}}^{2}}{8} \\
\frac{B^{0}}{4} & -\frac{A^{0}}{4} & \frac{{A^{0}}^{2} + {B^{0}}^{2}}{8} & \frac{{A^{0}}^{2} + {B^{0}}^{2}}{6}
\end{pmatrix}. $$
\end{theorem}

\begin{proof}
 Please refer to Appendix \ref{Appendix:2} for the detailed exposition of the proof.   
\end{proof}
 

%% file: SC_ANmultiple.tex
\section{Asymptotic results for multiple harmonic components model}
\label{sec:4}
In this section we present the results of strong consistency and asymptotic normality of the least absolute deviation estimators for a multi component ($p>1$) model. We have the following result for the strong consistency for a multi component model. 
\begin{theorem}\label{theorem:3}
   Under Assumption \ref{as1} and Assumption \ref{as2}, $\hat{\theta}_{T,S}$, the LAD estimator of $\theta_0$ for the model \ref{eqn:1}, is strongly consistent for $\theta_0$, i.e.  \:\: $$\hat{\theta}_{T,S} \xrightarrow{a.s.}  \hat{\theta}_0, \:\: \text{as} \:\: min\{T,S\} \rightarrow \infty$$.
\end{theorem}

\begin{proof}
    Please refer to Appendix \ref{Appendix:3}.
\end{proof}

We have the following result for asymptotic normality for the LAD estimators of a multi component model.
\begin{theorem}\label{theorem:4}
 Let $$R(\hat{\theta}_{T,S}) = \left( R_{1}(\hat{\theta}_{T,S;1}),\cdots, \\R_{q}(\hat{\theta}_{T,S;p})\right),$$ 
 where,
 \begin{align*}
 \begin{split}
 R_{r}(\hat{\theta}_{T,S;r}) & = \left(\sqrt{TS}(\hat{A}_{T,S;r}-A_{r}^{0}),\sqrt{TS}(\hat{B}_{T,S;r}-B_{r}^{0}), \right. \\
 & \qquad\qquad \hspace{-0.5cm} \left. T^{\frac{3}{2}}S^{\frac{1}{2}}(\hat{\lambda}_{T,S;r}-\lambda_{r}^{0}),S^{\frac{3}{2}}T^{\frac{1}{2}}(\hat{\mu}_{T,S;r}-\mu_{r}^{0})\right),\\
 \end{split} 
 \end{align*} 
 $r=1,\ldots, p$. Under Assumption \ref{as1} and Assumption \ref{as2}, $R(\hat{\theta}_{T,S})$  converges in distribution to $N_{4p}\left(0,\frac{1}{4g^{2}(0)}\Sigma^{-1}\right)$ as $min\{T,S\} \rightarrow \infty$; where $\Sigma = ((\Sigma_{m,n})), \: m,n \in \{1,\ldots, p\}$ and
\begin{equation*}
    \Sigma_{m,n} = 
    \begin{cases}
    0, & ~ \text{if}~ m\neq n\\\\
    \begin{pmatrix}
\frac{1}{2} & 0 & \frac{B_{m}^{0}}{4} & \frac{B_{m}^{0}}{4}\\
0 & \frac{1}{2} & -\frac{A_{m}^{0}}{4} & -\frac{A_{m}^{0}}{4}\\
\frac{B_{m}^{0}}{4} & -\frac{A_{m}^{0}}{4} & \frac{{A^{0}_{m}}^2 + {B^{0}_{m}}^2}{6} & \frac{{A^{0}_{m}}^2 + {B^{0}_{m}}^2}{8} \\
\frac{B_{m}^{0}}{4} & -\frac{A_{m}^{0}}{4} & \frac{{A^{0}_{m}}^2 + {B^{0}_{m}}^2}{8} & \frac{{A^{0}_{m}}^2 + {B^{0}_{m}}^2}{6}
\end{pmatrix}, & ~ \text{if}~ m=n.
    \end{cases}
\end{equation*}
\end{theorem}

\begin{proof}
The proof follows along the similar lines of the proof of Theorem \ref{theorem:2}. 
\end{proof}

%% file: Simulation.tex
\section{Simulation studies}
\label{sec:5}
In this section we present the  simulation studies performed to observe the performance of LAD estimators for different sample sizes $(T,S)$ for 2-dimensional chirp model and also to compare it's performance with the non-robust LSE under the presence of heavy tailed noise. We first consider a one component simulation model:
\begin{equation}
     y(t,s) = A^0 \cos(\lambda^0 t + \mu^0 s) + B^0 \sin(\lambda^0 t + \mu^0 s) + \epsilon(t,s)~~t\in[T], s\in[S]. \label{onecomp}
\end{equation}
We take the true values
for the model parameters as $A^0 = 2.4,B^0 = 1.4, \lambda^{0} = 0.4, \mu^{0} = 0.6$. In the simulations, we considered the error distributions to be (i) $N(0,0.1^2)$, (ii) a student's $t$ distribution with one degree of freedom and (iii) a slash normal distribution. Note that the second and the third noise scenarios represent heavy tailed noise.  Further, under the normal noise, the LSE is the MLE.  In order to find the LAD estimators and LSE, we have used Nelder-Mead downhill simplex algorithm for numerical optimization of the corresponding objective function. We have reported the average estimate (AE), mean square error(MSE) of the LAD estimators and LSE, over 1000 simulation runs.  For comparison, we also reported the theoretical asymptotic variance (AsyVar-LAD) of the LAD estimators, derived in the paper. The results for model (\ref{onecomp}) are presented in Tables (\ref{tab:1}--\ref{tab:3}).
We also performed simulations on the following 2-component model
\begin{equation}
     y(t,s) = \sum_{k=1}^2 \left(A^0_k \cos(\lambda^0_k t + \mu^0_k s) + B^0_k \sin(\lambda^0_k t + \mu^0_k s) \right)+ \epsilon(t,s)~~t\in[T], s\in[S]. \label{2comp}
\end{equation}
For the model (\ref{2comp}), we take $A_1^0=4.2$, $A_2^0=3.3$, $B_1^0=3.6$, $B_2^0=2.7$, $\lambda_1^0=1.1$, $\lambda_2^0=0.24$, $\mu_1^0=1.9$ and $\mu_2^0=0.36$.  Choice of noise distribution scenarios for (\ref{2comp}) are (i) $N(0,1)$, (ii) a student's $t$ distribution with one degree of freedom and (iii) a slash normal distribution.  We take the values of  $T, S$ as $25,50, 100, 200$ and $300$.  AE and MSE for the LAD and LSE over 1000 simulation runs are presented in Tables \ref{tab:4}--\ref{tab:9}. 
\begin{table}[htb]
    \centering
    \caption{Results for model (\ref{onecomp}) with $N(0,0.1^2)$ noise} \label{tab:1}
    \begin{tabular}{|c||c||c||c||c||c|}
    \hline
      \textbf{(T,S)} & \textbf{} & $A^0=2.4$ & $B^0=1.4$ & $\lambda^0=0.4$ & $\mu^0=0.6$ \\\hline
       (25,25) & LAD AE & 2.4 & 1.393 & 0.399 & 0.6\\ 
               & LAD MSE & 1.254E-04& 2.748E-04 & 1.253E-07& 1.143E-07\\
               & AsyVar-LAD & 1.268E-03 &2.753E-03 & 1.250E-06& 1.250E-06\\
               & LSE AE & 2.4 &1.399 &0.399 & 0.599\\
               & LSE MSE & 8.153E-05 & 1.835E-04 & 7.741E-08 &7.537E-08\\ \hline
       (50,50) & LAD AE &2.399 &1.399 & 0.4000 & 0.599 \\ 
               & LAD MSE & 3.330E-05 &7.190E-05 &7.635E-09 &7.988E-09\\
               & AsyVar-LAD & 3.171E-04 & 6.882E-04 &7.813E-08 &7.813E-08\\
               & LSE AE &2.399 & 1.399 & 0.4 & 0.6\\
               & LSE MSE &2.095E-05 &4.641E-05 &4.962E-09 &5.128E-09\\ \hline
       (75,75) & LAD AE & 2.399 & 1.4 & 0.399 & 0.6\\ 
               & LAD MSE &1.403E-05 & 3.136E-05 &1.534E-09 &1.520E-09\\
               & AsyVar-LAD & 1.409E-04 &3.059E-04 &1.543E-08 &1.543E-08\\
               & LSE AE & 2.4 & 1.399 & 0.399 & 0.6\\
               & LSE MSE &8.713E-06 &1.985E-05 &9.399E-10 &9.817E-10\\ \hline
         (150,150) & LAD AE & 2.4 & 1.4 & 0.4 & 0.599\\ 
               & LAD MSE &3.395E-06 &7.874E-06 &9.698E-11 &9.665E-11\\
               & AsyVar-LAD & 3.523E-05 &7.647E-05 & 9.646E-10 &9.646E-10\\
               & LSE AE & 2.399 & 1.4 & 0.4 & 0.6\\
               & LSE MSE &2.129E-06 &5.169E-06 &5.993E-11 &6.285E-11\\ \hline 
        (300,300) & LAD AE & 2.4 & 1.399 & 0.4 & 0.599\\ 
               & LAD MSE &8.303E-07 &1.712E-06 &5.708E-12 &5.643E-12\\
               & AsyVar-LAD & 8.808E-06 &1.912E-05 & 6.029E-11 &6.029E-11\\
               & LSE AE & 2.4 & 1.4 & 0.4 & 0.6\\
               & LSE MSE &5.007E-07 &1.178E-06 &3.63E-12 &3.521E-12\\ \hline 
    \end{tabular} 
     
\end{table}

\begin{table}[htb]
    \centering
    \caption{Results for model (\ref{onecomp}) with  $t_{(1)}$ noise} \label{tab:2}
    \begin{tabular}{|c||c||c||c||c||c|}
    \hline
      \textbf{(T,S)} & \textbf{} & $A^0=2.4$ & $B^0=1.4$ & $\lambda^0=0.4$ & $\mu^0=0.6$ \\\hline
       (25,25) & LAD AE & 2.395 & 1.385 & 0.399 & 0.6 \\ 
               & LAD MSE & 1.939E-02 & 4.059E-02 & 1.712E-05 & 1.855E-05\\
               & AsyVar-LAD & 1.992E-02 & 4.324E-02 &
               1.964E-05 &1.964E-05\\
               & LSE AE & 6.273 & 3.369 &0.129 & 0.806\\
               & LSE MSE & 959.629 & 2307.294 & 97.391 & 105.001\\ \hline
       (50,50) & LAD AE & 2.401 & 1.396 & 0.399 & 0.599 \\ 
               & LAD MSE & 5.246E-03 & 1.146E-02 & 1.209E-06 &1.182E-06\\
               & AsyVar-LAD &4.981E-03 &1.081E-02 & 1.227E-06 &
               1.227E-06\\
               & LSE AE & 7.637 & 2.815 & 0.995 & 6.908E-02\\
               & LSE MSE & 2041.593 & 1711.425 & 187.92 & 148.399\\ \hline
    (75,75) & LAD AE & 2.402 & 1.398 & 0.399 & 0.6 \\ 
               & LAD MSE & 2.381E-03 & 4.857E-03 & 2.690E-07 &2.329E-07\\
               & AsyVar-LAD &2.214E-03 &4.805E-03 &2.424E-07 &2.424E-07\\
               & LSE AE & 5.357 & 2.607 & 0.418 & 0.543\\
               & LSE MSE & 223.681 & 172.408 & 16.792 & 59.392\\ \hline
       (150,150) & LAD AE & 2.399 & 1.4 & 0.4 & 0.6 \\ 
               & LAD MSE & 5.638E-04 & 1.124E-03 & 1.398E-08 & 1.485E-08\\
               & AsyVar-LAD & 5.534E-04 & 1.201E-03 & 1.515E-08 &1.515E-08\\
               & LSE AE & 4.373 & 2.064 & 0.427 & 0.611\\
               & LSE MSE & 63.557 & 22.072 & 0.12 & 6.227E-02\\ \hline
        (300,300) & LAD AE & 2.4 & 1.4 & 0.4 & 0.6 \\ 
               & LAD MSE & 1.396E-04 & 2.94E-04 & 8.4E-10 & 9.406E-10\\
               & AsyVar-LAD & 1.384E-04 & 3.003E-04 & 9.469E-10 &9.469E-10\\
               & LSE AE & 3.069 & 1.588 & 0.418 & 0.614\\
               & LSE MSE & 11.437 & 3.13 & 0.009 & 0.01\\ \hline
    \end{tabular}
    
    \end{table}

\begin{table}[htb]
    \centering
    \caption{Results for model (\ref{onecomp}) with slash normal noise}
    \label{tab:3}
    \begin{tabular}{|c||c||c||c||c||c|}
    \hline
      \textbf{(T,S)} & \textbf{} & $A^0=2.4$ & $B^0=1.4$ & $\lambda^0=0.4$ & $\mu^0=0.6$ \\\hline
       (25,25) & LAD AE & 2.388 & 1.39 & 0.399 & 0.6 \\ 
               & LAD MSE & 5.264E-02&
               0.114 & 5.265E-05 &5.174E-05\\
               & AsyVar-LAD &5.073E-02 &0.110 & 5.001E-05 &5.001E-05\\
               & LSE AE & 9.535 & 3.18 & 0.878 & -0.696\\
               & LSE MSE & 2850.278 & 1607.34 & 2110.14 & 2268.6\\ \hline
       (50,50) & LAD AE & 2.399 & 1.389 & 0.399 & 0.6 \\ 
               & LAD MSE & 1.125E-02 &2.705E-02 &2.969E-06 &2.966E-06\\
               & AsyVar-LAD & 1.268E-02& 2.753E-02 &3.125E-06 &3.125E-06\\
               & LSE AE & 9.364 & 0.11 & 7.818E-02 & 0.583\\
               & LSE MSE & 6537.038 & 2596.877 & 267.709 & 51.612\\ \hline
       (75,75) & LAD AE & 2.398 & 1.399 & 0.4 & 0.599 \\ 
               & LAD MSE & 5.589E-03&1.187E-02 &6.196E-07 &5.757E-07\\
               & AsyVar-LAD &5.637E-03 &1.224E-02 &6.173E-07 &6.173E-07\\
               & LSE AE & 8.533 & 2.257 & 0.379 & 0.248\\
               & LSE MSE & 1449.823 & 224.178 & 55.204 & 154.336\\ \hline
       (150,150) & LAD AE & 2.401 & 1.399 & 0.399 & 0.599 \\ 
               & LAD MSE & 1.417E-03 &3.105E-03 &3.833E-08 &3.809E-08\\
               & AsyVar-LAD &1.409E-03 &3.059E-03 &3.858E-08 &3.858E-08\\
               & LSE AE & 4.598 & 2.238 & 0.421 & 0.572\\
               & LSE MSE & 45.089 & 32.407 & 0.131 & 1.923\\ \hline
     (300,300) & LAD AE & 2.399 & 1.4 & 0.4 & 0.6 \\ 
               & LAD MSE & 3.067E-04 &6.979E-04 &2.186E-09 &2.459E-09\\
               & AsyVar-LAD &3.523E-04 &7.647E-04 &2.412E-09 &2.412E-09\\
               & LSE AE & 3.270 & 1.838 & 0.423 & 0.616\\
               & LSE MSE & 8.740 & 6.705 & 0.012 & 0.013\\ \hline
    \end{tabular}
    \end{table}

\begin{table}[htb]
    \centering
    \caption{Frequency parameter results for model (\ref{2comp}) with  $N(0,1)$ noise}
    \label{tab:4}
    \begin{tabular}{|c||c||c||c||c||c|}
    \hline
      \textbf{(T,S)} & \textbf{} &  $\lambda_1^0=1.1$  & $\mu_1^0=1.9$ &$\lambda_2^0=0.24$ &$\mu_2^0=0.36$  \\  \hline 
       (25,25) & LAD AE &  1.100 	&  1.900 & 	0.240& 0.360  \\ 
               & LAD MSE & 2.566E-06  & 3.309E-06  &  6.067E-06 & 5.146E-06 \\
               
               & AsyVar-LAD & 3.153E-06& 3.153E-06 &5.308E-06 & 5.308E-06
               \\
               & LSE AE &1.100 	&  1.900 & 	0.240& 0.360  \\
               & LSE MSE & 1.877E-06  &  2.205E-06 &  3.644E-06 & 3.552E-06 \\ \hline 
       (50,50) & LAD AE &1.100 	&  1.900 & 	0.240& 0.360   \\ 
               & LAD MSE & 1.941E-07 & 1.934E-07  & 3.249E-07  & 3.771E-07 \\
               
               & AsyVar-LAD & 1.971E-07 & 1.971E-07 & 3.317E-07 & 3.317E-07
               \\
               & LSE AE &1.100 	&  1.900 & 	0.240& 0.360   \\
               & LSE MSE & 1.941E-07  & 1.934E-07  &  3.249E-07 & 3.771E-07 \\ \hline
    (100,100) &  LAD AE &1.100 	&  1.900 & 	0.240& 0.360     \\ 
               & LAD MSE & 1.250E-08  & 1.094E-08  & 2.144E-08  & 2.102E-08 \\
               
               & AsyVar-LAD & 1.231E-08&1.231E-08& 2.073E-08 &  2.073E-08
               \\
               & LSE AE &1.100 	&  1.900 & 	0.240& 0.360    \\
               & LSE MSE & 8.458E-09  & 7.553E-09  &  1.327E-08 & 1.267E-08 \\ \hline
       (200,200) &  LAD AE &1.100 	&  1.900 & 	0.240& 0.360  \\ 
               & LAD MSE & 8.327E-10  &  7.416E-10 & 1.329E-09  & 1.167E-09 \\
               
               & AsyVar-LAD & 7.699E-10& 7.699E-10 &1.296E-09&  1.296E-09
               \\
               & LSE AE &1.100 	&  1.900 & 	0.240& 0.360  \\
               & LSE MSE &  5.325E-10 &  4.871E-10 & 8.102E-10  & 6.948E-10 \\ \hline
    (300,300) &  LAD AE &   	1.100 	&  1.900 & 	0.240& 0.360   \\ 
               & LAD MSE & 1.507E-10  &  1.425E-10 &  2.514E-10 & 2.616E-10 \\
               
               & AsyVar-LAD &1.521E-10 &1.521E-10& 2.560E-10 & 2.560E-10
               \\
               & LSE AE &   1.100 	&  1.900 & 	0.240& 0.360  \\
               & LSE MSE & 9.756E-11  &  8.974E-11 & 1.423E-10  &  1.696E-10\\ \hline
    \end{tabular}
    
\end{table}

\begin{table}[htb]
    \centering
    \caption{Amplitude parameter results for model (\ref{2comp}) with  $N(0,1)$ noise}
    \label{tab:5}
    \begin{tabular}{|c||c||c||c||c||c|}
    \hline
      \textbf{(T,S)} & \textbf{} & $A_1^0=4.2$ & $B_1^0=3.6$ & $A_2^0=3.3$ & $B_2^0=2.7$ \\\hline
       (25,25) &  LAD AE &  4.206 	& 3.596  & 	3.290& 2.705  \\ 
               & LAD MSE & 1.551E-02  & 2.087E-02  &  1.755E-02 & 2.627E-02 \\
               
               & AsyVar-LAD & 1.779E-02 & 2.241E-02 & 1.712E-02 &2.309E-02
               \\
               & LSE AE &  4.207 & 3.598  &  3.293 & 2.706\\
               & LSE MSE & 1.109E-02  & 1.382E-02  &  1.235E-02 & 1.736E-02 \\ \hline
       (50,50) &  LAD AE &  4.202 	&  3.595 & 3.302	&  2.700 \\ 
               & LAD MSE &  4.283E-03 & 5.620E-03  & 4.770E-03  & 6.500E-03 \\
               
               & AsyVar-LAD & 4.449E-03 & 5.603E-03& 4.280E-03 & 5.773E-03
               \\
               & LSE AE & 4.202  & 3.597  &  3.302 & 2.700 \\
               & LSE MSE & 2.630E-03  & 2.964E-03  &  3.347E-03 & 3.856E-03 \\ \hline
          (100,100) &  LAD AE &   	4.199& 3.601  & 3.298	& 2.702  \\ 
               & LAD MSE & 1.128E-03  & 1.270E-03  & 1.118E-03  & 1.394E-03 \\
               
               & AsyVar-LAD & 1.112E-03& 1.401E-03& 1.070E-03 &1.443E-03
               \\
               & LSE AE &  4.200 & 3.600  &  3.298 &  2.701\\
               & LSE MSE & 7.567E-04  & 8.999E-04  &  6.842E-04 & 9.033E-04 \\ \hline
           (200,200) &  LAD AE &  4.200 	&  3.601 & 	3.300&  2.701 \\ 
               & LAD MSE &  2.578E-04 &  3.436E-04 & 2.530E-04  & 3.706E-04 \\
               
               & AsyVar-LAD & 2.781E-04 &3.501E-04 & 2.675E-04 & 3.608E-04
               \\
               & LSE AE &  4.201 & 3.600  &  3.300 & 2.700 \\
               & LSE MSE & 1.811E-04  & 2.262E-04  & 1.605E-04  & 2.240E-04 \\ \hline
               (300,300) &  LAD AE &  4.200 	& 3.600  & 	3.300&  2.700 \\ 
               & LAD MSE & 1.245E-04  & 1.528E-04  & 1.218E-04  &  1.603E-04\\
               
               & AsyVar-LAD & 1.236E-04& 1.556E-04& 1.188E-04 & 1.603E-04
               \\
               & LSE AE &   4.200 	& 3.600  & 	3.300&  2.700  \\
               & LSE MSE & 7.681E-05  &  1.023E-04 &  7.373E-05 & 9.455E-05 \\ \hline  
    \end{tabular}
        
\end{table}

\begin{table}[htb]
    \centering
    \caption{Frequency parameter results for model (\ref{2comp}) with  $t_{(1)}$ noise}
    \label{tab:6}
    \begin{tabular}{|c||c||c||c||c||c|}
    \hline
      \textbf{(T,S)} & \textbf{} &  $\lambda_1^0=1.1$  & $\mu_1^0=1.9$ &$\lambda_2^0=0.24$ &$\mu_2^0=0.36$ \\\hline
       (25,25) & LAD AE & 1.100	& 1.900 & 0.240	& 0.360 \\ 
       
               & LAD MSE &   5.395E-06  	& 4.760E-06 & 7.986E-06 &	8.023E-06 \\
               & AsyVar-LAD & 4.954E-06 & 4.954E-06 & 8.338E-06 &
               8.338E-06 \\
               & LSE AE & 0.984	 & 2.061 & 0.143 &0.396 \\
               & LSE MSE & 10.496 & 13.402 & 2.264 & 1.247 \\ \hline
       (50,50) & LAD AE &  1.100	& 1.900 &0.240	& 0.360  \\ 
               & LAD MSE & 2.741E-07 & 2.921E-07 & 5.110E-07 & 5.002E-07\\
               
               & AsyVar-LAD & 3.096E-07& 3.096E-07 & 5.211E-07 & 5.211E-07 
               \\
               & LSE AE & -2.271 & 3.483 & -0.102 & 0.527\\
               & LSE MSE & 5447.033 & 1206.998 & 44.346 & 9.943\\ \hline
    (100,100) & LAD AE &   1.100	& 1.900 &0.240	& 0.360   \\ 
               & LAD MSE & 2.051E-08  & 1.894E-08  &  3.471E-08 & 3.003E-08\\
               & AsyVar-LAD & 1.935E-08 & 1.935E-08 & 3.257E-08 & 3.257E-08 \\
               & LSE AE & 0.963  & 2.030  &  0.209 & 0.352 \\
               & LSE MSE & 2.938  &  4.519 &  0.473 & 0.709 \\ \hline
       (200,200) & LAD AE &  1.100	& 1.900 &0.240	& 0.360   \\ 
               & LAD MSE & 1.11E-09 &  1.25E-09  & 1.87E-09  & 2.19E-09\\
               & AsyVar-LAD & 7.699E-10 &  7.699E-10 &  1.296E-09 & 1.296E-09\\
               & LSE AE &  1.067 &  1.894 & 0.210  & 0.381 \\
               & LSE MSE & 12.693  & 19.110  & 0.556  & 0.602  \\ \hline
    (300,300) & LAD AE &   1.100	& 1.900 &0.240	& 0.360   \\ 
               & LAD MSE & 2.414E-10  &  2.372E-10 & 3.822E-10  & 4.567E-10 \\
               & AsyVar-LAD & 2.389E-10 & 2.389E-10  & 4.021E-10  & 4.021E-10 \\
               & LSE AE & 1.398  &  1.314 & 0.112  &  0.552\\
               & LSE MSE & 28.622  &  93.677 &  7.008 & 24.612 \\ \hline
    \end{tabular}
    
\end{table}

\begin{table}[htb]
    \centering
    \caption{Amplitude parameter results for model (\ref{2comp}) with  $t_{(1)}$ noise}
    \label{tab:7}
    \begin{tabular}{|c||c||c||c||c|c|}
    \hline
      \textbf{(T,S)} & \textbf{} & $A_1^0=4.2$ & $B_1^0=3.6$ & $A_2^0=3.3$ & $B_2^0=2.7$ \\\hline
       (25,25) & LAD AE &  4.193&3.602&3.298 & 2.694 \\ 
       
               & LAD MSE & 2.94E-02& 3.65E-02&2.78E-02& 3.43E-02  \\
               & AsyVar-LAD & 2.796E-02 &  3.531E-02& 2.689E-02
                &3.627E-02\\
               & LSE AE &  6.183 & 5.145  & 5.617  &  0.836\\
               & LSE MSE &  473.31  &  404.887 & 536.017  & 487.241  \\ \hline
       (50,50) & LAD AE & 4.205  & 3.592& 3.299 & 2.699  \\ 
       
               & LAD MSE & 6.213E-03 & 8.527E-03 &6.824E-03 &  8.532E-03 \\
               & AsyVar-LAD & 4.449E-03 & 5.603E-03 &4.280E-03
                & 5.773E-03\\
               & LSE AE &  19.509  &   -8.423 &  29.918  & -0.346  \\
               & LSE MSE & 8.21E+04    & 9.012E+04   &  2.954E+05  & 1.045E+03   \\ \hline
          (100,100) & LAD AE &  4.199  & 3.601 & 3.301  &  2.699  \\ 
       
               & LAD MSE &  1.771E-03 & 2.236E-03  & 1.797E-03 &  2.362E-03  \\
               & AsyVar-LAD & 1.747E-03 & 2.200E-03 & 1.681E-03
                & 2.267E-03\\
               & LSE AE &  7.937   & 6.727    &  5.519   &   -0.761 \\
               & LSE MSE &   1619.304   & 1970.075    &  708.635   & 3444.44    \\ \hline  
           (200,200) & LAD AE &  4.2  & 3.6 & 3.299  &  2.7 \\ 
       
               & LAD MSE &  4.730E-04 & 5.530E-04  & 4.570E-04 & 5.650E-04   \\
               & AsyVar-LAD & 4.368E-04 & 5.501E-04 & 4.202E-04
                & 5.667E-04\\
               & LSE AE &   8.885  &  7.377   &   4.806  &   4.806 \\
               & LSE MSE &   2483.897   &  3151.481   &  1151.603   & 2606.801    \\ \hline   
               (300,300) & LAD AE &    4.2  & 3.6 & 3.3  &  2.7  \\ 
       
               & LAD MSE &  2.039E-04 &  2.447E-04 & 1.647E-04 & 2.404E-04   \\
               & AsyVar-LAD & 1.941E-04 & 2.445E-04 & 1.867E-04
                & 2.519E-04\\
               & LSE AE &   13.656  &   2.466  &  16.266   &  -3.599  \\
               & LSE MSE &  10672.010    & 1511.217    &  79871.000   &   14605.760  \\ \hline   
    \end{tabular}
    
\end{table}

\begin{table}[htb]
    \centering
    \caption{Frequency parameter results for model (\ref{2comp}) with  Slash normal noise}
    \label{tab:8}
    \begin{tabular}{|c||c||c||c||c||c|}
    \hline
      \textbf{(T,S)} &  &  $\lambda_1^0=1.1$  & $\mu_1^0=1.9$ &$\lambda_2^0=0.24$ &$\mu_2^0=0.36$ \\\hline
       (25,25) & LAD AE & 1.100	& 1.900 &0.240	& 0.360   \\ 
               & LAD MSE & 1.179E-05  &  1.277E-05 & 2.055E-05  & 2.328E-05 \\
               
               & AsyVar-LAD &1.261E-05 & 1.261E-05&  2.123E-05 &2.123E-05
               \\
               & LSE AE & -23.417  & 5.911  & -1.365  & 0.817 \\
               & LSE MSE &  1.906E+05 & 5.109E+03  & 2.402E+03  & 3.079E+02 \\ \hline 
       (50,50) & LAD AE &   	1.100	& 1.900 &0.240	& 0.360   \\ 
               & LAD MSE &  7.688E-07 &  7.589E-07 & 1.294E-06  & 1.359E-06 \\
               
               & AsyVar-LAD &7.885E-07 &7.885E-07 & 1.327E-06 & 1.327E-06
               \\
               & LSE AE & 1.173  & 1.529  &  0.164 & 0.447 \\
               & LSE MSE &  19.244 & 95.621  &  1.699 & 8.863 \\ \hline
    (100,100) &  LAD AE &   	1.100	& 1.900 &0.240	& 0.360   \\ 
               & LAD MSE & 5.031E-08  & 5.046E-08  & 8.243E-08  & 8.721E-08 \\
               
               & AsyVar-LAD &4.927E-08 & 4.927E-08& 8.294E-08 & 8.294E-08
               \\
               & LSE AE & 1.169  & 1.679  & 0.171 & 0.439 \\
               & LSE MSE & 1.564  &  22.181 &  8.597E-01 & 1.593 \\ \hline
       (200,200) & LAD AE & 1.100	& 1.900 &0.240	& 0.360   \\ 
               & LAD MSE & 3.315E-09  & 3.279E-09  & 5.505E-09  & 5.285E-09 \\
               
               & AsyVar-LAD &3.079E-09& 3.079E-09 & 5.184E-09 & 5.184E-09
               \\
               & LSE AE & -0.622  & 3.747  &  0.479 & 0.071 \\
               & LSE MSE & 1605.535  & 1858.705  & 37.976  & 43.989 \\ \hline
    (300,300) &  LAD AE & 1.100	& 1.900 &0.240	& 0.360   \\ 
               & LAD MSE & 5.560E-10  & 5.786E-10 & 1.135E-09  & 1.172E-09 \\
               
               & AsyVar-LAD & 6.083E-10& 6.083E-10& 1.024E-09 & 1.024E-09
               \\
               & LSE AE & 0.679  &  2.169 &  0.179 & 0.345 \\
               & LSE MSE &  95.240 &  29.760 &  4.469 & 4.428 \\ \hline
    \end{tabular}
    
\end{table}

\begin{table}[htb]
    \centering
    \caption{Amplitude parameter results for model (\ref{2comp}) with  Slash normal noise}
    \label{tab:9}
    \begin{tabular}{|c||c||c||c||c|c|}
    \hline
      \textbf{(T,S)} & \textbf{} & $A_1^0=4.2$ & $B_1^0=3.6$ & $A_2^0=3.3$ & $B_2^0=2.7$ \\\hline
       (25,25) &  LAD AE & 4.167  	&  3.620 & 	3.285 & 2.697   \\ 
               & LAD MSE &  6.718E-02 & 7.931E-02  & 6.623E-02  & 9.391E-02 \\
               
               & AsyVar-LAD & 7.119E-02& 8.964E-02& 6.848E-02 &9.236E-02
               \\
               & LSE AE &  73.999 &  5.926 & 47.691  & 5.541 \\
               & LSE MSE &  1.250E+06 &  2.564E+03 & 8.136E+05  &  1.770E+05\\ \hline
       (50,50) &  LAD AE &  4.206 	&  3.592 & 	3.302 & 2.693  \\ 
               & LAD MSE &  1.830E-02 & 2.232E-02  & 1.675E-02  & 2.191E-02 \\
               
               & AsyVar-LAD & 1.779E-02& 2.241E-02& 1.712E-02 & 2.309E-02
               \\
               & LSE AE & 12.280  & 5.930  & 5.510  &  2.144\\
               & LSE MSE & 7999.326  & 3102.814  &  753.322 & 355.829 \\ \hline
          (100,100) &  LAD AE &   4.197	& 3.603  &3.296 	& 2.703  \\ 
               & LAD MSE &  4.889E-03 & 6.265E-03  &  4.016E-03 & 5.602E-03 \\
               
               & AsyVar-LAD & 4.449E-03& 5.603E-03& 4.280E-03 &5.773E-03
               \\
               & LSE AE & 6.150  & 4.310  &  5.695 & 3.572 \\
               & LSE MSE &  389.474 & 212.655  &  351.717 & 335.386 \\ \hline
           (200,200) &  LAD AE &   4.202	&  3.598 & 	3.300&  2.700 \\ 
               & LAD MSE &  1.128E-03 & 1.504E-03  & 1.049E-03  & 1.391E-03 \\
               
               & AsyVar-LAD &1.112E-03 & 1.401E-03& 1.070E-03 & 1.443E-03
               \\
               & LSE AE & 8.961  &  18.250 & 1.913  & 0.865 \\
               & LSE MSE & 1794.662  &  96567.081 &  3177.226 &  2449.394 \\ \hline
               (300,300) &  LAD AE &   	4.200&   3.600 & 	3.300&  2.698 \\ 
               & LAD MSE &  4.789E-04 &  6.489E-04 & 5.176E-04  & 7.241E-04 \\
               
               & AsyVar-LAD & 4.944E-04&6.225E-04 & 4.755E-04 & 6.414E-04
               \\
               & LSE AE & 11.527  &  8.708 & 6.386  & 4.103 \\
               & LSE MSE & 14157.988  &  4393.049 &  770.175 & 1300.654 \\ \hline  
    \end{tabular}
\end{table}

We also performed simulations for model \ref{2comp} under outlier contaminated datasets. With the noise distribution as normal, slash normal and Cauchy, we contaminate a pre-fixed proportion of data with outliers by adding a large constant to the generated data.  We report some representative results for a 20\% outlier contamination setup in Tables \ref{tab:10}--\ref{tab:14}.

\begin{table}[htb]
    \centering
    \caption{Frequency parameter results for model (\ref{2comp}) with  $t_{(1)}$ noise and 20\% outliers}
    \label{tab:10}
    \begin{tabular}{|c||c||c||c||c||c|}
    \hline
      \textbf{(T,S)} & \textbf{} &  $\lambda_1^0=1.1$  & $\mu_1^0=1.9$ &$\lambda_2^0=0.24$ &$\mu_2^0=0.36$ \\\hline
       (25,25) & LAD AE & 1.100& 1.900& 0.240& 0.360 \\ 
       
               & LAD MSE &   5.548E-06& 4.809E-06& 7.832E-06& 9.852E-06 \\
               & AsyVar-LAD & 4.954E-06 & 4.954E-06 & 8.338E-06 &
               8.338E-06 \\
               & LSE AE & -36.330& 31.863& -1.706& 1.868 \\
               & LSE MSE & 2.849E+05& 1.823E+05& 6.924E+02& 4.421E+02 \\ \hline
       (50,50) & LAD AE &  1.100 & 1.900 & 0.240 & 0.360  \\ 
               & LAD MSE & 2.823E-07 & 3.311E-07 & 5.871E-07 & 6.167E-07\\
               
               & AsyVar-LAD & 3.096E-07& 3.096E-07 & 5.211E-07 & 5.211E-07 
               \\
               & LSE AE & 0.955 & 1.983 & 0.136 & 0.402\\
               & LSE MSE & 1.933E+01 & 6.833 & 1.926 & 1.141\\ \hline
    (100,100) & LAD AE &   1.100 &  1.900 &0.240 &  0.360   \\ 
               & LAD MSE & 2.286E-08 &  2.410E-08  &  3.460E-08 &  3.673E-08\\
               & AsyVar-LAD & 1.935E-08 & 1.935E-08 & 3.257E-08 & 3.257E-08 \\
               & LSE AE & 0.985 &  1.817   &  0.229 &  0.348 \\
               & LSE MSE & 4.241 &  2.087 &  1.205E-01 &  2.310E-01 \\ \hline
       (200,200) & LAD AE &  1.100  &   1.900  &   0.240  &   0.360   \\ 
               & LAD MSE & 1.156E-09  &   9.964E-10  &   2.043E-09  &   2.312E-09\\
               & AsyVar-LAD & 7.699E-10 &  7.699E-10 &  1.296E-09 & 1.296E-09\\
               & LSE AE &  1.217  &   1.817  &   0.178  &   0.226 \\
               & LSE MSE & 1.359  &   6.422E-01  &   7.903E-01  &   5.856  \\ \hline
    (300,300) & LAD AE &   1.100  &  1.900  &  0.240  &  0.360   \\ 
               & LAD MSE & 2.927E-10  &  3.005E-10  &  4.656E-10  &  3.978E-10 \\
               & AsyVar-LAD & 2.389E-10 & 2.389E-10  & 4.021E-10  & 4.021E-10 \\
               & LSE AE & 0.373  &  2.209  &  -0.091  &  0.502\\
               & LSE MSE & 1.034E+02  &  1.827E+01  &  1.738E+01  &  3.104 \\ \hline
    \end{tabular}
    
\end{table}

\begin{table}[htb]
    \centering
    \caption{Frequency parameter results for (\ref{2comp}) with  Slash normal noise and 20\% outliers}
    \label{tab:11}
    \begin{tabular}{|c||c||c||c||c||c|}
    \hline
      \textbf{(T,S)} &  &  $\lambda_1^0=1.1$  & $\mu_1^0=1.9$ &$\lambda_2^0=0.24$ &$\mu_2^0=0.36$ \\\hline
       (25,25) & LAD AE & 1.100 & 1.900 & 0.241 & 0.359   \\ 
               & LAD MSE & 1.351E-05 & 1.240E-05 & 2.518E-05 & 2.066E-05 \\
               
               & AsyVar-LAD &1.261E-05 & 1.261E-05&  2.123E-05 &2.123E-05
               \\
               & LSE AE & 0.390 & 2.637 & 0.062 & 0.430 \\
               & LSE MSE &  3.007E+01 & 5.930E+01 & 1.720 & 6.610E-01 \\ \hline 
       (50,50) & LAD AE &   	1.100 & 1.900 & 0.240 & 0.360   \\ 
               & LAD MSE &  8.693E-07 & 8.163E-07 & 1.666E-06 & 1.306E-06 \\
               
               & AsyVar-LAD &7.885E-07 &7.885E-07 & 1.327E-06 & 1.327E-06
               \\
               & LSE AE & 0.922 & 0.868 & 0.229 & 0.260 \\
               & LSE MSE &  4.244 & 3.277E+02 & 2.015 & 2.556 \\ \hline
    (100,100) &  LAD AE &   	1.100	& 1.900 &0.240	& 0.360   \\ 
               & LAD MSE & 6.462E-08 & 5.168E-08 & 8.567E-08 & 9.212E-08 \\
               
               & AsyVar-LAD &4.927E-08 & 4.927E-08& 8.294E-08 & 8.294E-08
               \\
               & LSE AE & -11.145 & 15.646 & 0.442 & -0.179 \\
               & LSE MSE & 2.885E+04 & 4.034E+04 & 8.504E+01 & 9.629E+01 \\ \hline
       (200,200) & LAD AE & 1.100	& 1.900 &0.240	& 0.360   \\ 
               & LAD MSE & 3.221E-09 & 3.046E-09 & 5.204E-09 & 6.020E-09 \\
               
               & AsyVar-LAD &3.079E-09& 3.079E-09 & 5.184E-09 & 5.184E-09
               \\
               & LSE AE & 1.426 & 2.021 & 0.142 & 0.359 \\
               & LSE MSE & 2.847E+01 & 1.769E+01 & 2.239 & 8.404E-01 \\ \hline
    (300,300) &  LAD AE & 1.100	& 1.900 &0.240	& 0.360   \\ 
               & LAD MSE & 6.071E-10 & 6.583E-10 & 1.057E-09 & 1.110E-09 \\
               
               & AsyVar-LAD & 6.083E-10& 6.083E-10& 1.024E-09 & 1.024E-09
               \\
               & LSE AE & 0.355 & 2.407 & -0.659 & 0.974 \\
               & LSE MSE &  1.131E+02 & 5.333E+01 & 1.538E+02 & 7.245E+01 \\ \hline
    \end{tabular}
    
\end{table}

\begin{table}[htb]
    \centering
    \caption{Amplitude parameter results for  (\ref{2comp}) with  Slash normal noise and 20\% outliers}
    \label{tab:12}
    \begin{tabular}{|c||c||c||c||c|c|}
    \hline
      \textbf{(T,S)} & \textbf{} & $A_1^0=4.2$ & $B_1^0=3.6$ & $A_2^0=3.3$ & $B_2^0=2.7$ \\\hline
       (25,25) &  LAD AE & 4.198 & 3.607 & 3.274 & 2.706  \\ 
               & LAD MSE &  8.572E-02 & 1.149E-01 & 8.073E-02 & 1.051E-01 \\
               
               & AsyVar-LAD & 7.119E-02& 8.964E-02& 6.848E-02 &9.236E-02
               \\
               & LSE AE &  15.014 & 8.059 & 4.065 & 2.960 \\
               & LSE MSE &  1.171E+04 & 6.622E+02 & 8.425E+02 & 2.404E+02\\ \hline
       (50,50) &  LAD AE &  4.199 & 3.585 & 3.316 & 2.677  \\ 
               & LAD MSE &  2.121E-02 & 2.539E-02 & 1.987E-02 & 2.471E-02 \\
               
               & AsyVar-LAD & 1.779E-02& 2.241E-02& 1.712E-02 & 2.309E-02
               \\
               & LSE AE & 21.172 & 8.938 & 27.106 & -7.772\\
               & LSE MSE & 3.096E+04 & 3.964E+03 & 4.314E+04 & 6.117E+03 \\ \hline
          (100,100) &  LAD AE &   4.196 & 3.600 & 3.305 & 2.700  \\ 
               & LAD MSE &  4.892E-03 & 6.621E-03 & 4.125E-03 & 6.125E-03 \\
               
               & AsyVar-LAD & 4.449E-03& 5.603E-03& 4.280E-03 &5.773E-03
               \\
               & LSE AE & 47.570 & -8.122 & 16.908 & 5.798 \\
               & LSE MSE &  3.219E+05 & 2.718E+04 & 1.555E+04 & 2.192E+03 \\ \hline
           (200,200) &  LAD AE &   4.202 & 3.597 & 3.298 & 2.701 \\ 
               & LAD MSE &  1.262E-03 & 1.512E-03 & 1.124E-03 & 1.435E-03 \\
               
               & AsyVar-LAD &1.112E-03 & 1.401E-03& 1.070E-03 & 1.443E-03
               \\
               & LSE AE & 14.463 & 3.168 & 5.266 & -1.803  \\
               & LSE MSE & 8.883E+03 & 1.510E+02 & 3.717E+02 & 2.947E+03 \\ \hline
               (300,300) &  LAD AE &   	4.199 & 3.600 & 3.299 & 2.702 \\ 
               & LAD MSE &  4.931E-04 & 5.582E-04 & 4.782E-04 & 6.382E-04 \\
               
               & AsyVar-LAD & 4.944E-04&6.225E-04 & 4.755E-04 & 6.414E-04
               \\
               & LSE AE & 16.077 & 20.590 & 14.520 & -5.173 \\
               & LSE MSE & 2.102E+04 & 5.288E+04 & 1.651E+04 & 9.986E+03 \\ \hline  
    \end{tabular}
    
\end{table}

\begin{table}[htb]
    \centering
    \caption{Frequency parameter results for (\ref{2comp}) with  $N(0,9)$ noise and 20\% outliers}
    \label{tab:13}
    \begin{tabular}{|c||c||c||c||c||c|}
    \hline
      \textbf{(T,S)} & \textbf{} &  $\lambda_1^0=1.1$  & $\mu_1^0=1.9$ &$\lambda_2^0=0.24$ &$\mu_2^0=0.36$  \\  \hline 
       (25,25) & LAD AE &  1.100 & 1.900 & 0.242 & 0.360  \\ 
               & LAD MSE & 3.133E-05 & 3.011E-05 & 5.704E-05 & 4.275E-05 \\
               
               & AsyVar-LAD & 2.837E-05& 2.837E-05 &4.777E-05 & 4.777E-05
               \\
               & LSE AE &1.100 & 1.900 & 0.244 & 0.360  \\
               & LSE MSE & 7.807E-05 & 7.606E-05 & 3.005E-04 & 1.007E-04 \\ \hline 
       (50,50) & LAD AE &1.100 & 1.900 & 0.240 & 0.360   \\ 
               & LAD MSE & 1.952E-06 & 2.008E-06 & 3.317E-06 & 3.428E-06 \\
               
               & AsyVar-LAD & 1.773E-06 & 1.773E-06 & 2.985E-06 & 2.985E-06
               \\
               & LSE AE &1.100 & 1.900 & 0.240 & 0.360   \\
               & LSE MSE & 3.934E-06 & 3.786E-06 & 9.406E-06 & 7.790E-06 \\ \hline
    (100,100) &  LAD AE &1.100 	&  1.900 & 	0.240& 0.360     \\ 
               & LAD MSE & 1.136E-07 & 7.917E-08 & 2.060E-07 & 1.874E-07 \\
               
               & AsyVar-LAD & 1.107E-07&1.107E-07& 1.865E-07 &  1.865E-07
               \\
               & LSE AE &1.100 & 1.900 & 0.240 & 0.360    \\
               & LSE MSE & 2.450E-07 & 2.508E-07 & 5.109E-07 & 4.466E-07 \\ \hline
       (200,200) &  LAD AE &1.100 	&  1.900 & 	0.240& 0.360  \\ 
               & LAD MSE & 6.607E-09 & 8.360E-09 & 1.022E-08 & 9.466E-09 \\
               
               & AsyVar-LAD & 6.929E-09& 6.929E-09 &1.166E-08&  1.166E-08
               \\
               & LSE AE &1.100 	&  1.900 & 	0.240& 0.360  \\
               & LSE MSE &  1.553E-08 & 1.915E-08 & 2.901E-08 & 1.909E-08 \\ \hline
    (300,300) &  LAD AE &   	1.100 	&  1.900 & 	0.240& 0.360   \\ 
               & LAD MSE & 1.459E-09 & 1.342E-09 & 3.003E-09 & 2.448E-09 \\
               
               & AsyVar-LAD &1.368E-09 &1.368E-09& 2.304E-09 & 2.304E-09
               \\
               & LSE AE &   1.100 	&  1.900 & 	0.240& 0.360  \\
               & LSE MSE & 2.386E-09 & 3.372E-09 & 6.194E-09 & 5.396E-09\\ \hline
    \end{tabular}
    
\end{table}

\begin{table}[htb]
    \centering
    \caption{Amplitude parameter results for (\ref{2comp}) with  $N(0,9)$ noise and 20\% outliers}
    \label{tab:14}
    \begin{tabular}{|c||c||c||c||c||c|}
    \hline
      \textbf{(T,S)} & \textbf{} & $A_1^0=4.2$ & $B_1^0=3.6$ & $A_2^0=3.3$ & $B_2^0=2.7$ \\\hline
       (25,25) &  LAD AE &  4.160 & 3.598 & 3.168 & 2.774  \\ 
               & LAD MSE & 1.525E-01 & 2.078E-01 & 1.884E-01 & 2.114E-01 \\
               
               & AsyVar-LAD & 1.601E-01 & 2.0173E-01 & 1.541E-01 &2.078E-01
               \\
               & LSE AE &  4.155 & 3.517 & 2.979 & 2.823\\
               & LSE MSE & 4.749E-01 & 5.294E-01 & 7.062E-01 & 8.922E-01 \\ \hline
       (50,50) &  LAD AE &  4.176 & 3.602 & 3.300 & 2.711 \\ 
               & LAD MSE &  4.635E-02 & 5.127E-02 & 4.328E-02 & 5.048E-02  \\
               
               & AsyVar-LAD & 4.001E-02 & 5.043E-02& 3.852E-02 & 5.201E-02
               \\
               & LSE AE & 4.178 & 3.604 & 3.291 & 2.678 \\
               & LSE MSE & 1.069E-01 & 1.127E-01 & 1.147E-01 & 1.320E-01 \\ \hline
          (100,100) &  LAD AE &  4.201 & 3.601 & 3.296 & 2.705   \\ 
               & LAD MSE & 9.089E-03 & 1.294E-02 & 9.920E-03 & 1.251E-02 \\
               
               & AsyVar-LAD & 1.001E-02& 1.260E-02& 9.612E-03 &1.303E-02
               \\
               & LSE AE &  4.197 & 3.606 & 3.302 & 2.706\\
               & LSE MSE & 2.079E-02 & 3.194E-02 & 2.391E-02 & 2.787E-02  \\ \hline
           (200,200) &  LAD AE &  4.198 & 3.598 & 3.301 & 2.704 \\ 
               & LAD MSE &  2.542E-03 & 3.595E-03 & 1.993E-03 & 3.092E-03 \\
               
               & AsyVar-LAD & 2.517E-03 &3.234E-03 & 2.401E-03 & 3.245E-03
               \\
               & LSE AE &  4.199 & 3.598 & 3.299 & 2.701 \\
               & LSE MSE & 6.686E-03 & 7.399E-03 & 4.640E-03 & 6.109E-03\\ \hline
               (300,300) &  LAD AE &  4.198 & 3.603 & 3.296 & 2.701 \\ 
               & LAD MSE & 1.195E-03 & 1.609E-03 & 1.169E-03 & 1.567E-03 \\
               
               & AsyVar-LAD & 1.112E-04  & 1.400E-04& 1.069E-04 & 1.443E-04
               \\
               & LSE AE &   4.202 & 3.599 & 3.296 & 2.703 \\
               & LSE MSE & 2.158E-03 & 2.843E-03 & 2.664E-03 & 3.868E-03  \\ \hline  
    \end{tabular}
        
\end{table}

\clearpage

Simulation studies indicate satisfactory and robust performance of the proposed LAD estimates.  From the above tables, we can observe that as $(T,S)$ increases from (25,25) to (300,300), the mean square errors and the average biases of LAD estimators decrease. This is very much in line with the consistency properties of the LAD estimators. For heavy tailed error distributions like the $t$ distribution with 1 degree of freedom and the slash normal, clearly LAD estimators outperforms least square estimators in terms of both the average bias and mean square error, which indicates robust behaviour of the LAD estimates. Performance of the non-robust LSEs break down under the presence of heavy tailed noise and fails to identify the correct parameters, especially under small sample sizes.  Under Gaussian error LSEs are the MLEs and performs the best, as expected.  Performance of the LAD estimators under Gaussian error is comparable with the best performing LSEs. Further, we observe that  asymptotic variance and the mean square errors of the LAD estimators become quite comparable as $(T,S)$ increases, indicating the validity of the asymptotic distribution results.  We also observe that for all the scenarios, the performance of LAD estimates of frequencies are far more accurate than the estimates of amplitudes, which is expected as the frequencies have a higher rate of convergence.

%% file: DataAnalysis.tex
\section{Synthetic data analysis}\label{sec:6}
In this section, we analyse a synthetically generated data using model (\ref{eqn:1}) to demonstrate how the proposed parameter estimation method works. The synthetic data is generated using the similar model structure and parameters as outlined in (\ref{onecomp}). We assumed $T = S = 100$ and the true model parameters were $(A^{0}, B^{0}, \lambda^{0}, \mu^{0}) = (2.4, 1,4, 0.6, 0.4)$. The noisy texture, where the noise component is drawn from slash normal distribution, is plotted in Figure \ref{fig:1}; and the original texture (without the noise component), is plotted in Figure \ref{fig:2}. Our problem is to extract the original texture Figure \ref{fig:2} from the noisy texture Figure \ref{fig:1}. We performed the least absolute deviation based estimation method for the parameter estimation.  The LAD estimated parameters are given by, $$\hat{A} = 2.3444058, \hat{B} = 1.4133243, \hat{\lambda} =  0.4000882, \hat{\mu} = 0.5999380$$
Based on the above estimates, the estimated grayscale texture is plotted in Figure \ref{fig:3}. From Figure \ref{fig:2} and Figure \ref{fig:3} it is quite clear that they match quite well. However, the estimated texture obtained from least square estimation fails to recover the texture from the noisy image data.

\begin{figure}[ht]
\caption{Grey scale image of noisy synthetic true data}
    \label{fig:1}
    \centering
    \includegraphics[width=10cm]{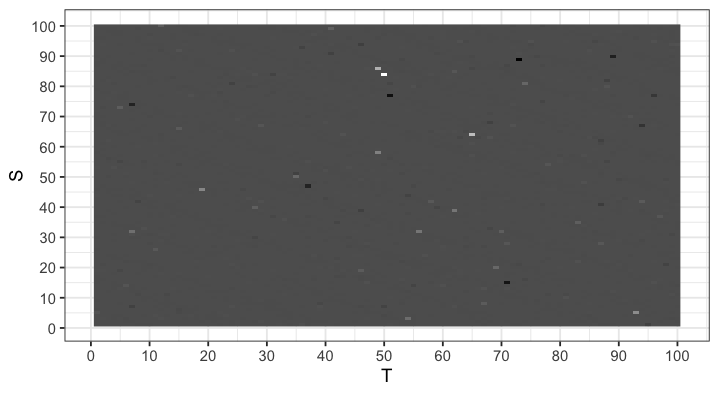}
\end{figure}

\begin{figure}[htb]
\caption{Grey scale image of Noiseless synthetic true data}
    \label{fig:2}
    \centering
    \includegraphics[width=10cm]{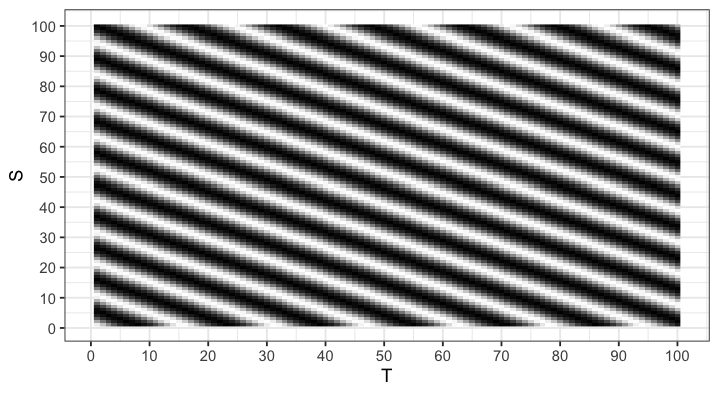}
\end{figure}

\begin{figure}[htb]
\caption{Grey scale image of recovered data from LAD estimation}
    \label{fig:3}
    \centering
    \includegraphics[width=10cm]{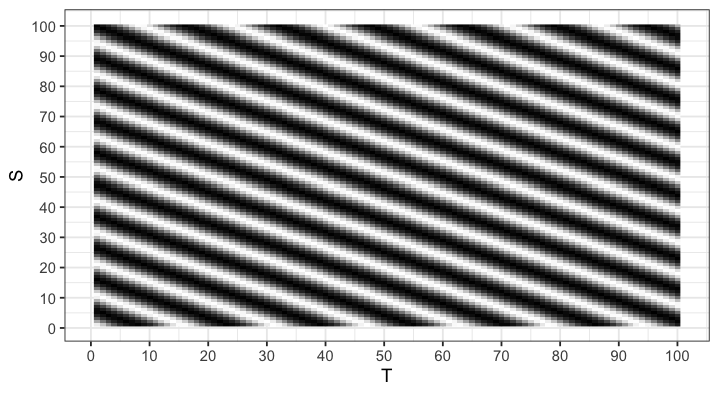}
\end{figure}

%% file: Conclusions.tex
\section{Conclusion} \label{sec:7}
In this paper, we proposed a robust least absolute deviation estimators for estimating the parameters of a 2-dimensional sinusoidal model. We have established that the LAD estimators are  strongly consistent and further have asymptotic normal distribution. It is observed that the LAD estimators of frequencies have a higher rate of convergence than the LAD estimators of amplitudes.  It is observed in the simulations that in case of heavy tailed noise distribution and cases where outliers are present, the LAD method of estimation clearly outperforms least square method of estimation.  While the performance of the non-robust LSEs break down under the presence of heavy tailed noise and fails to identify the correct parameters, the proposed LAD estimators provide robust estimates and performs satisfactorily under heavy tailed noise distributions.  Synthetic data analysis of a texture data provides evidence of practical applicability of proposed LAD approach.

% \section*{Acknowledgement} 
% 	The work of the second author is partially supported by grant number MTR/2020/ 000599 of Science  \& Engineering Research  Board, Department of Science \& Technology, Government of India.

%\section*{Declarations} 

%Conflict of interests: The authors declare that they have no conflict of interests.\\
%Competing interests: The authors declare that they have no competing interests. \\
%Data availability: The simulated datasets generated during the current study are available from the corresponding author on reasonable request.

%% file: appendix.tex
%\newpage 
\appendix
\section{Appendix}
\subsection{Proof of Thereom 1}\label{Appendix:1}
For all  $ \theta \neq \theta_{0} $, define,
  \begin{equation*}
      H_{T,S}(\theta) = Q_{T,S}(\theta) - Q_{T,S}(\theta_{0})
  \end{equation*}
  Now expanding $H_{T,S}(\theta)$ explicitly, we obtain,
  \begin{align*}
  \begin{split}
      H_{T,S}(\theta) &= \frac{1}{TS}\sum\limits_{t=1}^{T}\sum\limits_{s=1}^{S}\Bigg\{\Big\vert y(t,s) - A \cos(\lambda t + \mu s) - B \sin(\lambda t + \mu s)\Big\vert - {} \\
  &\qquad\qquad \Big\vert y(t,s) - A^{0}\cos(\lambda^{0}t + \mu^{0}s) - B^{0}\sin(\lambda^{0}t + \mu^{0}s) \Big\vert\Bigg\}
  \end{split}
  \notag\\[1ex]
  &= \frac{1}{TS}\sum\limits_{t=1}^{T}\sum\limits_{s=1}^{S} \Bigg\{\Big\vert h_{t,s}(\theta)+\epsilon(t,s)\Big\vert-\Big\vert\epsilon(t,s)\Big\vert\Bigg\} = \frac{1}{TS}\sum\limits_{t=1}^{T}\sum\limits_{s=1}^{S} Z_{t,s}(\theta)
  \end{align*}
  where, $Z_{t,s} = \Bigg\{\Big\vert h_{t,s}(\theta)+\epsilon(t,s)\Big\vert-\Big\vert\epsilon(t,s)\Big\vert\Bigg\}$  and $h_{t,s}(\theta) = A^{0}\cos(\lambda^{0}t + \mu^{0}s) + B^{0}\sin(\lambda^{0}t + \mu^{0}s) - A \cos(\lambda t + \mu s) - B \sin(\lambda t + \mu s)$. $Z_{t,s}$'s being only a function of \text{i.i.d} $\epsilon(t,s)$ are also independently distributed random variables. The proof of the theorem proceeds through essentially the two following important steps: \begin{itemize}
      \item $H_{T,S}(\theta)-\lim\limits_{T,S \rightarrow \infty}\mathbb{E}(H_{T,S}(\theta)) \overset{a.s.}\longrightarrow 0 ~~~ \text{uniformly for all}~\theta \in \Theta$ \\
      \item $\theta_{0}$ is the unique minimizer of $\lim\limits_{T,S \rightarrow \infty}\mathbb{E}(H_{T,S}(\theta))$.
  \end{itemize}
  We first show that $\mathbb{E}\Big[Z_{t,s}(\theta)\Big]<\infty$. Let us consider the case when $h_{t,s}(\theta)>0$.
  \begin{align*}
      \mathbb{E}\Big[Z_{t,s}(\theta)\Big] &= \int Z_{t,s}(\theta)~ dG(\epsilon(t,s)) \notag\\[1ex]
      &= \int_{\epsilon(t,s)>0}h_{t,s}(\theta)dG(\epsilon(t,s)) + \int_{-h_{t,s}(\theta)}^{0}\Big(h_{t,s}(\theta)+2\epsilon(t,s)\Big)dG(\epsilon(t,s))  {} \\
  &\qquad\qquad \hspace{5cm} - \int\limits_{-\infty}^{-h_{t,s}(\theta)}h_{t,s}(\theta)dG(\epsilon(t,s))
      \notag\\[1ex]
      &= 2\int_{-h_{t,s}(\theta)}^{0}\Big(h_{t,s}(\theta)+\epsilon(t,s)\Big)dG(\epsilon(t,s)) ~~(\text{since}~G(0) = \frac{1}{2})
      \notag\\[1ex]
      &= 2\Big[h_{t,s}(\theta)G(0) - G\left(-h_{t,s}^{*}(\theta)\right)h_{t,s}(\theta)\Big]~~~\left(\text{with} -h_{t,s}^{*}(\theta)\in\left(0,-h_{t,s}(\theta)\right)\right).
  \end{align*}
  The last equality in the above is obtained by applying integration by parts and integral mean value theorem. Similarly, for $h_{t,s}(\theta)<0$ proceeding same as above, we obtain, \begin{align*}
     \mathbb{E}\Big[Z_{t,s}(\theta)\Big] = 2\Big[h_{t,s}(\theta)G(0) - G\left(-h_{t,s}^{**}(\theta)\right)h_{t,s}(\theta)\Big],~~\text{where}, -h_{t,s}^{**}(\theta)\in\left(0,-h_{t,s}(\theta)\right).
  \end{align*}
  Hence, it can be easily observed that $\mathbb{E}\Big[Z_{t,s}(\theta)\Big] < \infty$. Proceeding similarly as above we can show that \text{Var}$\Big[Z_{t,s}(\theta)\Big]<\infty$ and the bounds for both $\mathbb{E}\Big[Z_{t,s}(\theta)\Big]$ and Var$\Big[Z_{t,s}(\theta)\Big]$ are independent of $t$ and $s$. \\
  
  Further, since $\Theta$ is compact $ \exists~ \Theta_1,\ldots,\Theta_k$ such that $\Theta = \cup_{i=1}^{k} \Theta_i$ and $\sup\limits_{\theta \in \Theta_i} Z_{t,s}(\theta) - \inf\limits_{\theta \in \Theta_i}  Z_{t,s}(\theta) < \frac{\epsilon}{2^t 2^s}$ for each $\Theta_i$. Now,
  
  \begin{align*}
  \begin{split}
      H_{T,S}(\theta)-\lim_{T,S \rightarrow \infty}\mathbb{E}(H_{T,S}(\theta)) &= \Bigg[\frac{1}{TS}\sum\limits_{t=1}^{T}\sum\limits_{s=1}^{S} Z_{t,s}(\theta) - \frac{1}{TS}\sum\limits_{t=1}^{T}\sum\limits_{s=1}^{S}\mathbb{E}\sup\limits_{\theta\in\Theta_i}Z_{t,s}(\theta)\Bigg]+ {} \\
  &\qquad\qquad \hspace{-0.5cm}\Bigg[ \frac{1}{TS}\sum\limits_{t=1}^{T}\sum\limits_{s=1}^{S}\mathbb{E}\sup\limits_{\theta\in\Theta_i}Z_{t,s}(\theta) - \lim_{T,S \rightarrow \infty} \frac{1}{TS}\sum\limits_{t=1}^{T}\sum\limits_{s=1}^{S}\mathbb{E}Z_{t,s}(\theta)\Bigg].
  \end{split}
  \end{align*}
  We can show that the quantity $Z_{t,s}(\theta)$ is bounded by some quantity $M$ by using triangle inequality. Now using the compactness of the the parameter space and Kolmogorov's strong law of large numbers(\cite{bill}), we can show that
  $$H_{T,S}(\theta)-\lim_{T,S \rightarrow \infty}\mathbb{E}(H_{T,S}(\theta)) \overset{as.}\longrightarrow 0 ~~~ \text{uniformly for all}~\theta \in \Theta.$$
  Now define $Q(\theta) = \lim_{T,S \rightarrow \infty}\mathbb{E}(H_{T,S}(\theta))$. It is easy to observe that $Q(\theta_{0})=0$. For all $\theta \neq \theta_{0}$, we obtain,
  \begin{align*}
      \begin{split}
         Q(\theta) &= \lim_{T,S \rightarrow \infty}\frac{2}{TS}\sum\limits_{Z_{t,s}(\theta)>0}\Big[h_{t,s}(\theta)\{G(0) - G\left(-h_{t,s}^{*}(\theta)\right)\}\Big] + {}\\
         &\qquad\qquad \hspace{3cm}\frac{2}{TS}\sum\limits_{Z_{t,s}(\theta)<0}\Big[h_{t,s}(\theta)\{G(0) - G\left(-h_{t,s}^{**}(\theta)\right)\}\Big]
      \end{split}
      \notag\\[1ex]
      &\geq \lim_{T,S \rightarrow \infty}\frac{2}{TS}\sum\limits_{t=1}^{T}\sum\limits_{s=1}^{S}\big\vert h_{t,s}(\theta)\big\vert\min\bigg\{G\left(-h_{t,s}^{**}(\theta)\right)-\frac{1}{2}, \frac{1}{2} - G\left(-h_{t,s}^{*}(\theta)\right)\bigg\}.
  \end{align*}
  Now from Lemma \ref{lemma:1} and \cite[Lemma 4]{oberhofer}, it can be shown that $\lim\limits_{T,S \rightarrow \infty}\frac{1}{TS}\sum\limits_{t=1}^{T}\sum\limits_{s=1}^{S}\big\vert h_{t,s}(\theta)\big\vert^2 = \frac{1}{2}A^{2}_{0} + \frac{1}{2}B^{2}_{0} + \frac{1}{2}A^{2} + \frac{1}{2}B^{2}>0$. Hence, we can conclude that $Q(\theta)$ has an unique minimizer in $\theta_0\in\Theta$. Thus, by using the  sufficient conditions of Lemma 2.2 of \cite{white}, we can conclude that $\hat{\theta}_{T,S}$ is strongly consistent for $\theta_0$, the true value of the parameter vector $\theta$.

\subsection{Proof of Theorem 2}\label{Appendix:2}
Let us consider,
\begin{align*}
    Q_{T,S}(\theta) &= \frac{1}{TS}\sum\limits_{t=1}^{T}\sum\limits_{s=1}^{S}\Big\vert y(t,s) - A cos(\lambda t + \mu s) - B sin(\lambda t + \mu s)\Big\vert\\
    &= \frac{1}{TS}\sum\limits_{t=1}^{T}\sum\limits_{s=1}^{S}\big\vert h_{t,s}(\theta)+\epsilon(t,s)\big\vert,
\end{align*}
where, $h_{t,s}(\theta) = A^{0}cos(\lambda^{0}t + \mu^{0}s) + B^{0}sin(\lambda^{0}t + \mu^{0}s) - A cos(\lambda t + \mu s) - B sin(\lambda t + \mu s)$. Note that $Q_{T,S}(\theta)$ is not a differentiable function. Since we shall make use of Taylor series expansion for finding the asymptotic normality of LAD estimators, we approximate the function $Q_{T,S}(\theta)$ by a function $Q^{*}_{T,S}(\theta)$, where,
$$Q^{*}_{T,S}(\theta) = \frac{1}{TS}\sum\limits_{t=1}^{T}\sum\limits_{s=1}^{S}\delta_{T,S}(h_{t,s}(\theta)+\epsilon(t,s)),$$ such that,
$\lim_{x\rightarrow \infty}\delta_{T,S}(x) = |x|$ and $\delta_{T,S}(x)$ be a smooth function of $x$.
We take
\begin{align*}
    \delta_{T,S}(x) &= \bigg[-\frac{1}{3}\beta^{2}_{T,S}x^3 + \beta_{T,S}x^2 + \frac{1}{3\beta_{T,S}}\bigg]\mathbb{I}_{\{0<x\leq \frac{1}{\beta_{T,S}}\}} + x\mathbb{I}_{\{x>\frac{1}{\beta_{T,S}}\}},
\end{align*}
where, $\mathbb{I}_{B}$ denote the indicator function and the following conditions are satisfied:
\begin{itemize}
    \item $\delta_{T,S}(x)$ is an even function of $x$, i.e. $\delta_{T,S}(x) = \delta_{T,S}(-x)$.
    \item $\beta_{T,S}$ is an appropriately chosen increasing function of both $T$ and $S$, simultaneously, such that $\lim\limits_{T,S\rightarrow\infty}\frac{1}{\beta_{T,S}} = 0$ with $T^2 S^2 = o(\beta^3_{T,S})$ and $\beta_{T,S} = o(TS)$.
\end{itemize}
Note that, unlike $Q_{T,S}(\theta)$, $Q^{*}_{T,S}(\theta)$ is a twice continuously differentiable function. Let $\theta^{'}_{T,S} = (A^{'}_{T,S},B^{'}_{T,S},\lambda^{'}_{T,S},\mu^{'}_{T,S})$ denote the  minimizer of $Q^{*}_{T,S}(\theta)$. It can be shown easily, following the arguments of  Appendix Section \ref{Appendix:1}, that $\theta^{'}_{T,S}$ is strongly consistent for $\theta_0$.
We first show that the minimum of $Q^{*}_{T,S}(\theta)$ is close to the minimum of $Q_{T,S}(\theta)$, for large $T,S$. To show this, let us consider the term $Q^{*}_{T,S}(\theta) - Q_{T,S}(\theta) $. Now,
\begin{align*}
    Q^{*}_{T,S}(\theta) - Q_{T,S}(\theta) &= \frac{1}{TS}\sum\limits_{t=1}^{T}\sum\limits_{s=1}^{S}\bigg\{\delta_{T,S}(h_{t,s}(\theta)+\epsilon(t,s)) - \vert h_{t,s}(\theta)+\epsilon(t,s)\vert\bigg\}\\
    &= \begin{cases} \frac{1}{TS}\sum\limits_{t=1}^{T}\sum\limits_{s=1}^{S}\big[ -\frac{1}{3}\beta^{2}_{T,S}u^3_{t,s}(\theta) + \beta_{T,S}u^2_{t,s}(\theta) + \frac{1}{3\beta_{T,S}}-u_{t,s}(\theta)\big], &{}\\
         \qquad\qquad \hspace{5.0cm} 0<u_{t,s}(\theta)\leq \frac{1}{\beta_{T,S}}\\
    \frac{1}{TS}\sum\limits_{t=1}^{T}\sum\limits_{s=1}^{S}\big[~~ \frac{1}{3}\beta^{2}_{T,S}u^3_{t,s}(\theta) + \beta_{T,S}u^2_{t,s}(\theta) + \frac{1}{3\beta_{T,S}}+u_{t,s}(\theta)\big], & {}\\
         \qquad\qquad \hspace{5.0cm}-\frac{1}{\beta_{T,S}}<u_{t,s}(\theta)\leq 0
    \end{cases}\\
    \begin{split}
    &= \frac{1}{TS}\sum\limits_{t=1}^{T}\sum\limits_{s=1}^{S}\bigg[ -\frac{1}{3}\beta^{2}_{T,S}u^3_{t,s}(\theta) + \beta_{T,S}u^2_{t,s}(\theta) + 
         \frac{1}{3\beta_{T,S}}-u_{t,s}(\theta)\bigg]{}\\
         &\qquad\qquad\mathbb{I}_{\{0<u_{t,s}(\theta)\leq \frac{1}{\beta_{T,S}}\}} + 
         \bigg[~~ \frac{1}{3}\beta^{2}_{T,S}u^3_{t,s}(\theta) + \beta_{T,S}u^2_{t,s}(\theta) + {}\\
         &\qquad\qquad\qquad\qquad\qquad\qquad \frac{1}{3\beta_{T,S}}+u_{t,s}(\theta)\bigg]\mathbb{I}_{\{-\frac{1}{\beta_{T,S}}<u_{t,s}(\theta)\leq 0\}}
    \end{split}\\
    &= \frac{1}{TS}\sum\limits_{t=1}^{T}\sum\limits_{s=1}^{S}W_{t,s}(\theta),\text{say,}
\end{align*}
where, $u_{t,s}(\theta) = h_{t,s}(\theta)+\epsilon(t,s))$.
Note that,
\begin{align*}
    \vert W_{t,s}(\theta)\vert &\leq \left(2\beta_{T,S}u^{2}_{t,s}(\theta) + \frac{2}{3\beta_{T,S}}\right)\mathbb{I}_{\{\vert u_{t,s}(\theta)\vert\leq \frac{1}{\beta_{T,S}}\}}\\
    &\leq \left(\frac{2}{\beta_{T,S}}+\frac{2}{3\beta_{T,S}}\right)\mathbb{I}_{\{\vert u_{t,s}(\theta)\vert \leq \frac{1}{\beta_{T,S}}\}}\\
    &= \frac{8}{3\beta_{T,S}}\mathbb{I}_{\{\vert u_{t,s}(\theta)\vert\leq \frac{1}{\beta_{T,S}}\}}.
\end{align*}
Thus we have,
\begin{align*}
    \mathbb{P}\left[\vert TS\left\{Q^{*}_{T,S}(\theta) - Q_{T,S}(\theta)\right\}\vert\geq \epsilon\right] &\leq \frac{\mathbb{E}\left\vert TS\left\{Q^{*}_{T,S}(\theta) - Q_{T,S}(\theta)\right\}\right\vert}{\epsilon}~~~~\\
    &= \frac{\mathbb{E}\left\vert\sum\limits_{t=1}^{T}\sum\limits_{s=1}^{S}W_{t,s}(\theta)\right\vert}{\epsilon}\\
    &\leq \frac{\sum\limits_{t=1}^{T}\sum\limits_{s=1}^{S}\mathbb{E}\left\vert W_{t,s}(\theta)\right\vert}{\epsilon}~~~~\\
    &\leq\frac{8}{3\beta_{T,S}\epsilon}\sum\limits_{t=1}^{T}\sum\limits_{s=1}^{S}\mathbb{P}\left\{\left\vert u_{t,s}(\theta)\right\vert\leq \frac{1}{\beta_{T,S}}\right\}\\
    &= \frac{8}{3\beta_{T,S}\epsilon}\sum\limits_{t=1}^{T}\sum\limits_{s=1}^{S}\int\limits_{-h_{t,s}(\theta)-\frac{1}{\beta_{T,S}}}^{-h_{t,s}(\theta)+\frac{1}{\beta_{T,S}}}g(\epsilon(t,s))d\epsilon(t,s)\\
    &= \frac{16~\text{g}(c)~TS}{3\epsilon\beta^{2}_{T,S}} ~~\text{(using integral mean value theorem)} \\
    &= \left(\frac{16~\text{g}(c)}{3\epsilon}\right)\left(\frac{\beta_{T,S}}{TS}\right)\left(\frac{T^{2}S^{2}}{\beta^{3}_{T,S}}\right)\rightarrow 0~~\text{as}~T,S\rightarrow \infty.
\end{align*}
Thus, we can conclude that $TS\left\{Q^{*}_{T,S}(\theta) - Q_{T,S}(\theta)\right\} = o_{p}(1)$. Now since the parameter space $\Theta$ is compact and $Q^{*}_{T,S}(\theta) - Q_{T,S}(\theta)$ is continuous, we can say that there exists $\theta^{*}$ such that $TS\left\{Q^{*}_{T,S}(\theta^{*}) - Q_{T,S}(\theta^{*})\right\} = \sup\limits_{\theta\in\Theta}TS\left\{Q^{*}_{T,S}(\theta) - Q_{T,S}(\theta)\right\}$. Hence we obtain from above,
\begin{equation}
\sup\limits_{\theta\in\Theta}TS\left\{Q^{*}_{T,S}(\theta) - Q_{T,S}(\theta)\right\} = o_{p}(1).\label{eq:5}
\end{equation}
Now $\hat{\theta}_{T,S}$ being the minimizer of $Q_{T,S}(\theta)$,  $Q_{T,S}(\hat{\theta}_{T,S})\leq Q_{T,S}(\theta^{'}_{T,S})$. Hence from \eqref{eq:5}, we can conclude that, \begin{equation}
    TS\left\{Q^{*}_{T,S}(\hat{\theta}_{T,S}) - Q^{*}_{T,S}(\theta^{'}_{T,S})\right\} = o_{p}(1).\label{eq:6}
\end{equation}
Now by Taylor series expansion upto $2^{\text{nd}}$ order, we have the following,
\begin{align}
\begin{split}
   Q^{*}_{T,S}(\hat{\theta}_{T,S}) =  Q^{*}_{T,S}(\theta^{'}_{T,S}) + \nabla Q^{*}_{T,S}(\theta^{'}_{T,S})(\hat{\theta}_{T,S}-\theta^{'}_{T,S}) + {}\\
         \qquad \hspace{2cm}\frac{1}{2}(\hat{\theta}_{T,S}-\theta^{'}_{T,S})^{T}\nabla^{2}Q^{*}_{T,S}(\bar{\theta}_{T,S})(\hat{\theta}_{T,S}-\theta^{'}_{T,S});\label{eq:7}
\end{split}         
\end{align}
where, $\bar{\theta}_{T,S} = \eta \hat{\theta}_{T,S}+ (1-\eta)\theta^{'}_{T,S}$ and $\eta \in (0,1)$ and $\nabla Q^{*}_{T,S}(.)$ and $\nabla^{2}Q^{*}_{T,S}(.)$ denote the first and second order derivatives of $Q^{*}_{T,S}(.)$.  Now since $\nabla Q^{*}_{T,S}(\theta^{'}_{T,S}) = 0$, \eqref{eq:7} gives
\begin{equation}
   Q^{*}_{T,S}(\hat{\theta}_{T,S}) =  Q^{*}_{T,S}(\theta^{'}_{T,S}) + \frac{1}{2}(\hat{\theta}_{T,S}-\theta^{'}_{T,S})^{T}\nabla^{2}Q^{*}_{T,S}(\bar{\theta}_{T,S})(\hat{\theta}_{T,S}-\theta^{'}_{T,S}).\label{eq:8}
\end{equation}
Realize that $\nabla^{2}Q^{*}_{T,S}(\bar{\theta}_{T,S})$ is a symmetric matrix and let  $\lambda$ denote the smallest eigenvalue of the matrix $\nabla^{2}Q^{*}_{T,S}(\bar{\theta}_{T,S})$. Now using min-max theorem, we have, $$(\hat{\theta}_{T,S}-\theta^{'}_{T,S})^{T}\nabla^{2}Q^{*}_{T,S}(\bar{\theta}_{T,S})(\hat{\theta}_{T,S}-\theta^{'}_{T,S}) \geq \lambda(\hat{\theta}_{T,S}-\theta^{'}_{T,S})^{T}(\hat{\theta}_{T,S}-\theta^{'}_{T,S}).$$
Now from \eqref{eq:8} we get
\begin{equation*}
    TS(\hat{\theta}_{T,S}-\theta^{'}_{T,S})^{T}(\hat{\theta}_{T,S}-\theta^{'}_{T,S}) \leq \frac{2TS\left\{Q^{*}_{T,S}(\hat{\theta}_{T,S}) -  Q^{*}_{T,S}(\theta^{'}_{T,S})\right\}}{\lambda}.
\end{equation*}
Now to prove that $\sqrt{TS}\|\hat{\theta}_{T,S}-\theta^{'}_{T,S}\| = o_{p}(1)$, it is sufficient to show that $\lambda>0$ as we already have  \eqref{eq:6}.
Thus it is enough to show that the matrix $\nabla^{2}Q^{*}_{T,S}(\bar{\theta}_{T,S})$ is a positive definite, as $ T,S \longrightarrow \infty$.\\\\
Let us now derive a result which will be useful to show the above stated positive definiteness.
Note that 
\begin{align*}
    \begin{split}
    \mathbb{E}\left[\frac{1}{TS}\sum\limits_{t=1}^{T}\sum\limits_{s=1}^{S}\delta^{''}_{T,S}(\epsilon(t,s))\right] &= \frac{1}{TS}\sum\limits_{t=1}^{T}\sum\limits_{s=1}^{S} \bigg[\int_{0}^{\frac{1}{\beta_{T,S}}}(-2\beta^{2}_{T,S}x+2\beta_{T,S})\, dG(x) + {}\\
         &\qquad \int_{-\frac{1}{\beta_{T,S}}}^{0}(2\beta^{2}_{T,S}x+2\beta_{T,S})\, dG(x)\bigg]\end{split}\\
    &= \frac{2}{TS}\sum\limits_{t=1}^{T}\sum\limits_{s=1}^{S}\left[g\left(\frac{1}{\beta_{T,S}}\right)-M\right]\\
    &= 2g\left(\frac{1}{\beta_{T,S}}\right)-4M~~\longrightarrow~2g(0) ~~\text{as}~~ T,S \longrightarrow \infty;
\end{align*}
where, $$M = \int_{0}^{\frac{1}{\beta_{T,S}}}g^{'}(x)\left(\beta^{2}_{T,S}\frac{x^2}{2}+\beta_{T,S}x\right)dx$$ 
and $$ |M|\leq C\int_{0}^{\frac{1}{\beta_{T,S}}}\left(\beta^{2}_{T,S}\frac{x^2}{2}+\beta_{T,S}x\right)dx = C\frac{2}{3\beta_{T,S}}\rightarrow 0~~ \text{as}~~T,S\rightarrow \infty.$$ Therefore \begin{equation*}
    \mathbb{E}\left[\frac{1}{TS}\sum\limits_{t=1}^{T}\sum\limits_{s=1}^{S}\delta^{''}_{T,S}(\epsilon(t,s))\right] = 2g(0) + o(1). \label{eq:9}
\end{equation*}
Similarly, we can show that $\text{Vr}\left[\frac{1}{TS}\sum\limits_{t=1}^{T}\sum\limits_{s=1}^{S}\delta^{''}_{T,S}(\epsilon(t,s))\right] = o(1)$. Combining both of them and by using Markov's inequality, we obtain that $\frac{1}{TS}\sum\limits_{t=1}^{T}\sum\limits_{s=1}^{S}\delta^{''}_{T,S}(\epsilon(t,s)) = 2g(0) + o_{p}(1)$.
\newline Now let $\nabla^{2}Q^{*}_{T,S}(\bar{\theta}_{T,S}) = ((q_{i,j}))_{4\times4}$. Using the lemma 3.2, straightforward calculations yields the following:
$$q_{1,1} = q_{2,2} = g(0) + o_{p}(1); ~~~~~~ q_{1,2} = o_p(1);$$
$$\frac{1}{T}q_{1,3} = \frac{1}{2}B^0 g(0) + o_{p}(1);~~~~~\frac{1}{S}q_{1,4} = \frac{1}{2}B^0 g(0) + o_{p}(1);$$
$$\frac{1}{T} q_{2,3} = -\frac{1}{2} A^0 g(0) + o_{p}(1);~~~~~\frac{1}{S} q_{2,4} = - \frac{1}{2}A^0 g(0) + o_{p}(1);$$

$$\frac{1}{T^2}q_{3,3} =  \frac{1}{S^2}q_{4,4} = \frac{1}{3}({A^{0}}^2 + {B^{0}}^{2}) + o_p(1); ~~~~~~\frac{1}{TS}q_{3,4} = \frac{1}{4}({A^{0}}^2 + {B^{0}}^2) + o_p(1).$$
Thus we have $\nabla^{2}Q^{*}_{T,S}(\bar{\theta}_{T,S})$  positive definite for $T,S \rightarrow \infty$. Further, we have  \begin{equation*}\sqrt{TS}\|\hat{\theta}_{T,S}-\theta^{'}_{T,S}\|= o_{p}(1). \label{eq:10}\end{equation*}
Under the same conditions, by using the Taylor series expansion of $Q^{*}_{T,S}(\theta)$ about only $\lambda$, we have 
$$Q^{*}_{T,S}(\hat{\theta}_{T,S}) - Q^{*}_{T,S}(\theta^{'}_{T,S}) = \frac{1}{2}(\hat{\lambda}_{T,S} -\lambda^{'}_{T,S})^{2}\frac{\partial^{2}Q^{*}_{T,S}(\bar{\theta}_{T,S})}{\partial \lambda^{2}}.$$ 
We have that
$ T^{3}S(\hat{\lambda}_{T,S} -\lambda^{'}_{T,S})^{2} = TS\left(Q^{*}_{T,S}(\hat{\theta}_{T,S}) - Q^{*}_{T,S}(\theta^{'}_{T,S})\right) + o_p(1)$.  And from \eqref{eq:6}, we get
\begin{equation*}
   T^{\frac{3}{2}}S^{\frac{1}{2}}(\hat{\lambda}_{T,S} -\lambda^{'}_{T,S}) = o_p(1); \label{eq:11} 
\end{equation*}
and \begin{equation*}S^{\frac{3}{2}}T^{\frac{1}{2}}(\hat{\mu}_{T,S} -\mu^{'}_{T,S}) = o_p(1).\end{equation*} 
Since $Q_{TS}^{*}(\theta)$ is minimized at $\theta=\theta_{TS}^{'}$; we apply multivariate mean value theorem on first derivative of $Q_{TS}^{*}(\theta)$ to get
\begin{align*}
  \begin{split}
    (Q^{*}_{T,S})_{A^0}  = -(Q^{*}_{T,S})_{\bar{A}\bar{A}}(A^{'}_{T,S}-A^0) + (Q^{*}_{T,S})_{\bar{A}\bar{B}}( B^{'}_{T,S}-B^{0}) + {}\\  (Q^{*}_{T,S})_{\bar{A}\bar{\lambda}}(\lambda^{'}_{T,S}-\lambda^0) + (Q^{*}_{T,S})_{\bar{A}\bar{\mu}}( \mu^{'}_{T,S}-\mu^0);
  \end{split}
\end{align*}
\begin{align*}
  \begin{split}
    (Q^{*}_{T,S})_{B^0}  = -(Q^{*}_{T,S})_{\bar{B}\bar{A}}( A^{'}_{T,S} - A^0) + (Q^{*}_{T,S})_{\bar{B}\bar{B}}(B^{'}_{T,S}-B^{0}) + {}\\ (Q^{*}_{T,S})_{\bar{B}\bar{\lambda}}(\lambda^{'}_{T,S}-\lambda^{0}) + (Q^{*}_{T,S})_{\bar{B}\bar{\mu}}(\mu^{'}_{T,S}-\mu^0);
  \end{split}
\end{align*}
\begin{align*}
 \begin{split}
    (Q^{*}_{T,S})_{\lambda^0}  = -(Q^{*}_{T,S})_{\bar{\lambda}\bar{A}}(A^{'}_{T,S}-A^{0}) + (Q^{*}_{T,S})_{\bar{\lambda}\bar{B}}(B^{'}_{T,S}-B^{0}) + {}\\(Q^{*}_{T,S})_{\bar{\lambda}\bar{\lambda}}(\lambda^{'}_{T,S}-\lambda^{0}) + (Q^{*}_{T,S})_{\bar{\lambda}\bar{\mu}}(\mu^{'}_{T,S}-\mu^{0});
 \end{split}
\end{align*}
\begin{align*}
 \begin{split}
    (Q^{*}_{T,S})_{\mu^0}  = -(Q^{*}_{T,S})_{\bar{\mu}\bar{A}}(A^{'}_{T,S}-A^{0}) + (Q^{*}_{T,S})_{\bar{\mu}\bar{B}}(B^{'}_{T,S}-B^{0}) + {}\\ \hspace{2cm} (Q^{*}_{T,S})_{\bar{\mu}\bar{\lambda}}(\lambda^{'}_{T,S}-\lambda^{0}) + (Q^{*}_{T,S})_{\bar{\mu}\bar{\mu}}(\mu^{'}_{T,S}-\mu^{0}).
 \end{split}
\end{align*}
Where,
$$(Q^{*}_{T,S})_{A^0} = \frac{\partial Q^{*}_{T,S}(\theta)}{\partial A}\bigg\vert_{(A^0,B^0,\lambda^0,\mu^0)}~~\text{and } ~~(Q^{*}_{T,S})_{\bar{A}\bar{B}} = \frac{\partial^{2}Q^{*}_{T,S}(\theta)}{\partial A\partial B}\bigg\vert_{(\bar{A}_{T,S},\bar{B}_{T,S},\bar{\lambda}_{T,S},\bar{\mu}_{T,S})}.$$
\\\\Note that here although the point $\bar{\theta}_{T,S} = (\bar{A}_{T,S},\bar{B}_{T,S},\bar{\lambda}_{T,S},\bar{\mu}_{T,S})$ is not same as that of $\bar{\theta}_{T,S}$ mentioned in \eqref{eq:7} but in order to avoid extra notations, we adopt the same notation here to convey the same meaning as that of \eqref{eq:7}. 
\newline Thus, we have 
\begin{align*}
\begin{split}
    \left(\sqrt{TS}(Q_{TS}^{*})_{A^0},\sqrt{TS}(Q_{TS}^{*})_{B^0},\frac{1}{\sqrt{TS}}(Q_{TS}^{*})_{\lambda^0},\frac{1}{\sqrt{TS}}(Q_{TS}^{*})_{\mu^0}\right) &= {}\\ 
    \left(\sqrt{TS}(A^{'}_{T,S}-A^0),\sqrt{TS}(B^{'}_{T,S}-B^0),T^{\frac{3}{2}}S^{\frac{1}{2}}(\lambda^{'}_{T,S}-\lambda^0),S^{\frac{3}{2}}T^{\frac{1}{2}}(\mu^{'}_{T,S}-\mu^0)\right)Z_{T,S};
\end{split} \label{eq:13}   
\end{align*}
where,\begin{equation*}
    Z_{T,S} = \begin{pmatrix}
    (Q^{*}_{T,S})_{\bar{A}\bar{A}} & (Q^{*}_{T,S})_{\bar{B}\bar{A}} & T^{-1}S^{-1}(Q^{*}_{T,S})_{\bar{\lambda}\bar{A}} & T^{-1}S^{-1}(Q^{*}_{T,S})_{\bar{\mu}\bar{A}}\\
    (Q^{*}_{T,S})_{\bar{A}\bar{B}}  & (Q^{*}_{T,S})_{\bar{B}\bar{B}} & T^{-1}S^{-1}(Q^{*}_{T,S})_{\bar{\lambda}\bar{B}} & 
    T^{-1}S^{-1}(Q^{*}_{T,S})_{\bar{\mu}\bar{B}}\\
    T^{-1}(Q^{*}_{T,S})_{\bar{A}\bar{\lambda}} & 
    T^{-1}(Q^{*}_{T,S})_{\bar{B}\bar{\lambda}} & 
    T^{-2}S^{-1}(Q^{*}_{T,S})_{\bar{\lambda}\bar{\lambda}} & 
    T^{-2}S^{-1}(Q^{*}_{T,S})_{\bar{\mu}\bar{\lambda}}\\
    S^{-1}(Q^{*}_{T,S})_{\bar{A}\bar{\mu}} & 
    S^{-1}(Q^{*}_{T,S})_{\bar{B}\bar{\mu}} & 
    S^{-2}T^{-1}(Q^{*}_{T,S})_{\bar{\lambda}\bar{\mu}} & 
    S^{-2}T^{-1}(Q^{*}_{T,S})_{\bar{\mu}\bar{\mu}}
    \end{pmatrix}.
\end{equation*}
Now we know that $\delta^{'}_{T,S}(x) = \left[-\beta^{2}_{T,S}x^{2} + 2\beta_{T,S}x\right]\mathbb{I}_{\left\{0<x\leq \frac{1}{\beta_{T,S}}\right\}} + \mathbb{I}_{\left\{x\geq \frac{1}{\beta_{T,S}}\right\}}~\text{and}~ \delta^{'}_{T,S}(-x) = -\delta^{'}_{T,S}(x) $. Hence by applying Markov's inequality it can be easily shown that,\begin{equation}
    \sqrt{TS}(Q_{TS}^{*})_{A^0} = \frac{1}{\sqrt{TS}}\sum\limits_{t=1}^{T}\sum\limits_{s=1}^{S}(-\cos(\lambda t + \mu s))\left(\mathbb{I}_{\{\epsilon(t,s)\geq\frac{1}{\beta_{T,S}}\}}-\mathbb{I}_{\{\epsilon(t,s)\leq-\frac{1}{\beta_{T,S}}\}}\right) + o_p(1)\label{eq:14}
\end{equation}
\begin{equation}
        \sqrt{TS}(Q_{TS}^{*})_{B^0} = \frac{1}{\sqrt{TS}}\sum\limits_{t=1}^{T}\sum\limits_{s=1}^{S}(-\sin(\lambda t + \mu s))\left(\mathbb{I}_{\{\epsilon(t,s)\geq\frac{1}{\beta_{T,S}}\}}-\mathbb{I}_{\{\epsilon(t,s)\leq-\frac{1}{\beta_{T,S}}\}}\right) + o_p(1)\label{eq:15}
\end{equation}
\begin{align}
\begin{split}
       \frac{1}{\sqrt{TS}}(Q_{TS}^{*})_{\lambda^0} = \frac{1}{\sqrt{T^{3}S}}\sum\limits_{t=1}^{T}\sum\limits_{s=1}^{S}(A^{0}t\sin(\lambda t + \mu s) - B^{0}t\cos(\lambda t + \mu s)){}\\\left(\mathbb{I}_{\{\epsilon(t,s)\geq\frac{1}{\beta_{T,S}}\}}-\mathbb{I}_{\{\epsilon(t,s)\leq-\frac{1}{\beta_{T,S}}\}}\right) + o_p(1)\label{eq:16}
\end{split}
\end{align}
\begin{align}
\begin{split}
       \frac{1}{\sqrt{TS}}(Q_{TS}^{*})_{\mu^0} = \frac{1}{\sqrt{S^{3}T}}\sum\limits_{t=1}^{T}\sum\limits_{s=1}^{S}(A^{0}s\sin(\lambda t + \mu s) - B^{0}s\cos(\lambda t + \mu s)){}\\ \left(\mathbb{I}_{\{\epsilon(t,s)\geq\frac{1}{\beta_{T,S}}\}}-\mathbb{I}_{\{\epsilon(t,s)\leq-\frac{1}{\beta_{T,S}}\}}\right) + o_p(1)\label{eq:17}
\end{split}
\end{align}\\

The sums \eqref{eq:14}-\eqref{eq:17} are of the form $\sum\limits_{t=1}^{T}\sum\limits_{s=1}^{S}U_{t,s}$, $U_{t,s}$ appropriately defined as in the corresponding equation. For \eqref{eq:14} we have,
\begin{eqnarray*}
       \mathbb{E}(U_{t,s}) &=& \sum\limits_{t=1}^{T}\sum\limits_{s=1}^{S}(-\cos(\lambda t + \mu s))\mathbb{E}\left(\mathbb{I}_{\{\epsilon(t,s)\geq\frac{1}{\beta_{T,S}}\}}-\mathbb{I}_{\{\epsilon(t,s)\leq-\frac{1}{\beta_{T,S}}\}}\right)\\
       &=& \sum\limits_{t=1}^{T}\sum\limits_{s=1}^{S}(-\cos(\lambda t + \mu s))\left(\int\limits_{\frac{1}{\beta_{T,S}}}^{\infty}d(G(\epsilon(t,s)))-\int\limits^{-\frac{1}{\beta_{T,S}}}_{-\infty}d(G(\epsilon(t,s)))\right)\\
       &=& \sum\limits_{t=1}^{T}\sum\limits_{s=1}^{S}(-\cos(\lambda t + \mu s))\left(1-G\left(\frac{1}{\beta_{T,S}}\right)-G\left(-\frac{1}{\beta_{T,S}}\right)\right).
\end{eqnarray*}
As $T,S\longrightarrow\infty$, we have $\frac{1}{\beta_{T,S}}\longrightarrow 0$; therefore, $\mathbb{E}(U_{t,s}) = o(1)$. Similarly for the equations \eqref{eq:15}-\eqref{eq:17}, we can show that, 
\begin{equation}\mathbb{E}(U_{t,s}) = o(1).\label{eq:18}\end{equation}

Now proceeding similarly as above we can show that,\begin{align*}
    \text{Var}\left(\sum\limits_{t=1}^{T}\sum\limits_{s=1}^{S}U_{t,s}\right) &= \begin{cases}
        \frac{1}{2} + o(1), & \text{for \eqref{eq:14},\eqref{eq:15}}\\\\
        \frac{{A^{0}}^{2}+{B^{0}}^{2}}{6} + o(1), & \text{for \eqref{eq:16},\eqref{eq:17}}
    \end{cases}\\
    &= B_{T,S},~~\text{(say)}.
\end{align*}

Further, we can show that for the above random variable $U_{t,s}$, the Lindeberg condition, i.e. $\lim\limits_{T,S\rightarrow\infty}\frac{1}{B_{T,S}}\sum\limits_{t=1}^{T}\sum\limits_{s=1}^{S}\mathbb{E}\left(U_{t,s}^{2}\mathbb{I}_{\left\{|U_{T,S}|\geq \epsilon \sqrt{ B_{T,S}}\right\}}\right) = 0$ holds and therefore by applying Lindeberg-Feller Central theorem (\cite{vaart}) we can conclude that, \\\\$\sqrt{TS}(Q_{TS}^{*})_{A^0}$, $\sqrt{TS}(Q_{TS}^{*})_{B^0}, \frac{1}{\sqrt{TS}}(Q_{TS}^{*})_{\lambda^0}$ and $\frac{1}{\sqrt{TS}}(Q_{TS}^{*})_{\mu^0}$ converges in distribution to $N(0,1/2)$, $N(0,1/2)$,  $N\left(0,\frac{{A^{0}}^{2} + {B^{0}}^{2}}{6}\right)$ and  $N\left(0,\frac{{A^{0}}^{2} + {B^{0}}^{2}}{6}\right)$, respectively.\\\\
Now we make use of the Cramer-Wald device to find the asymptotic joint distribution of $\left(\sqrt{TS}(Q_{TS}^{*})_{A^0},\sqrt{TS}(Q_{TS}^{*})_{B^0},\frac{1}{\sqrt{TS}}(Q_{TS}^{*})_{\lambda^0},\frac{1}{\sqrt{TS}}(Q_{TS}^{*})_{\mu^0}\right)$. Consider the linear combination, \begin{equation}
    A_{T,S} = a_1 (Q_{TS}^{'})_{A^0} + a_2 (Q_{TS}^{'})_{B^0} + a_3 (Q_{TS}^{'})_{\lambda^0} + a_4 (Q_{TS}^{'})_{\mu^0}; \label{eqn:ATS}
\end{equation}
where, $a_1,a_2,a_3~\text{and}~a_4$ are arbitrary real numbers. Now using \eqref{eq:18},  we have for \ref{eqn:ATS}, $$\mathbb{E}(A_{T,S}) =o(1)$$ and by using Lemma \ref{lemma:1} we can easily derive that
\begin{align*}
\begin{split}
\text{Var}(A_{T,S}) = \frac{a^{2}_{1}}{2} + \frac{a^{2}_{2}}{2} + \frac{{A^{0}}^{2} + {B^{0}}^{2}}{6} a^{2}_{3}+ \frac{{A^{0}}^{2} + {B^{0}}^{2}}{6} a^{2}_{4} + {}\\\qquad\frac{B^{0}}{2}a_{1}a_{3} + \frac{B^{0}}{2}a_{1}a_{4} + \frac{A^{0}}{2}a_{2}a_{3} +\frac{A^{0}}{2}a_{2}a_{4} + \frac{{A^{0}}^{2} + {B^{0}}^{2}}{4}a_{3}a_{4} + o(1).
\end{split}
\end{align*}

We can express the  random variable $A_{T,S}$ as $\sum\limits_{t=1}^{T}\sum\limits_{s=1}^{S}U_{t,s}$ and using the Lindeberg condition it is possible to apply the central limit theorem to $A_{T,S}$. Further, by applying the Cramer-Wald device we obtain that,\\
$\left(\sqrt{TS}(Q_{TS}^{*})_{A^0},\sqrt{TS}(Q_{TS}^{*})_{B^0},\frac{1}{\sqrt{TS}}(Q_{TS}^{*})_{\lambda^0},\frac{1}{\sqrt{TS}}(Q_{TS}^{*})_{\mu^0}\right)Z^{-1}_{T,S}$ converges in distribution to $N_4(0,\Sigma)$; where $$\Sigma = \begin{pmatrix}
\frac{1}{2} & 0 & \frac{B^{0}}{4} & \frac{B^{0}}{4}\\\\
0 & \frac{1}{2} & -\frac{A^{0}}{4} & -\frac{A^{0}}{4}\\\\
\frac{B^{0}}{4} & -\frac{A^{0}}{4} & \frac{{A^{0}}^{2} + {B^{0}}^{2}}{6} & \frac{{A^{0}}^{2} + {B^{0}}^{2}}{8} \\\\
\frac{B^{0}}{4} & -\frac{A^{0}}{4} & \frac{{A^{0}}^{2} + {B^{0}}^{2}}{8} & \frac{{A^{0}}^{2} + {B^{0}}^{2}}{6}
\end{pmatrix}.$$
It is easy to show that $\lim\limits_{T,S\rightarrow\infty}Z_{T,S} = 2g(0)\Sigma$. 
\newline Now,
\begin{align*}
    \begin{split}
        \left(\sqrt{TS}(\hat{A}_{T,S}-A^0),\sqrt{TS}(\hat{B}_{T,S}-B^0),T^{\frac{3}{2}}S^{\frac{1}{2}}(\hat{\lambda}_{T,S}-\lambda^0),S^{\frac{3}{2}}T^{\frac{1}{2}}(\hat{\mu}^{'}_{T,S}-\mu^0)\right) &= { } \text{  \:\:       } \\ \qquad \qquad  \underbrace{\left(\sqrt{TS}(A^{'}_{T,S}-A^0),\sqrt{TS}(B^{'}_{T,S}-B^0),T^{\frac{3}{2}}S^{\frac{1}{2}}(\lambda^{'}_{T,S}-\lambda^0),S^{\frac{3}{2}}T^{\frac{1}{2}}(\mu^{'}_{T,S}-\mu^0)\right)}_{A} + {}\\ \qquad\qquad
        \underbrace{\left(\sqrt{TS}(\hat{A}-A^{'}_{T,S}),\sqrt{TS}(\hat{B}-B^{'}_{T,S}),T^{\frac{3}{2}}S^{\frac{1}{2}}(\hat{\lambda}-\lambda^{'}_{T,S}),S^{\frac{3}{2}}T^{\frac{1}{2}}(\hat{\mu}-\mu^{'}_{T,S})\right)}_{B}
    \end{split}
\end{align*}
Now (B) goes to $0$ in probability and (A) has the same asymptotic distribution as that of $\left(\sqrt{TS}(Q_{TS}^{*})_{A^0},\sqrt{TS}(Q_{TS}^{*})_{B^0},\frac{1}{\sqrt{TS}}(Q_{TS}^{*})_{\lambda^0},\frac{1}{\sqrt{TS}}(Q_{TS}^{*})_{\mu^0}\right)Z_{T,S}^{-1}$, i.e. has asymptotic distribution as  $N_4\left(0,\frac{1}{4g^{2}(0)}\Sigma^{-1}\right)$. Applying the Slutsky's lemma, we obtain the desired result.  

\subsection{Proof of Theorem 3} \label{Appendix:3}
Let $\hat{\theta}^{4p\times 1}_{T,S} = (\hat{A}_{T,S;1},\hat{B}_{T,S;1},\hat{\lambda}_{T,S;1},\hat{\mu}_{T,S;1},\dots,\hat{A}_{T,S;p},\hat{B}_{T,S;p},\hat{\lambda}_{T,S;p},\hat{\mu}_{T,S;p})^{'}$ be the LAD estimator obtained by minimizing the quantity given by \eqref{eq:2}. Proceeding similarly as in the proof of Theorem \ref{theorem:1}, we can show that, $$H_{T,S}(\theta) - \lim\limits_{T,S\rightarrow\infty}\mathbb{E}(H_{T,S}(\theta)) \overset{\text{a.s.}}\longrightarrow 0~~~ \text{uniformly}~~\forall~~\theta\in\Theta,$$ and also that the true value $\theta_0 $ is the unique global minimizer of $\lim\limits_{T,S\rightarrow\infty}\mathbb{E}(H_{T,S}(\theta))$. Hence, from Lemma 2.2 of \cite{white}, strong consistency of  $\hat{\theta}_{T,S}$ follows.